\documentclass[11pt]{article}
\usepackage[utf8]{inputenc}

\usepackage{amsmath,amsthm,amssymb}
\usepackage{graphicx}
\usepackage{wrapfig}
\usepackage{enumitem}
\usepackage{tikz-cd}
\usepackage{esvect}
\usepackage{bm}
\usepackage{tocloft}
\usepackage{lipsum} 
\usepackage{mathrsfs}
\usepackage{float}
\usepackage{graphicx}
\usepackage{commath}
\usepackage{mathtools}
\usepackage{ulem}
\usepackage[nottoc]{tocbibind}

\font\sml=cmr7

\usetikzlibrary{backgrounds, matrix,fit,matrix,decorations.pathreplacing, calc, positioning}

\makeatletter
\newcommand\setItemnumber[1]{\setcounter{enum\romannumeral\@enumdepth}{\numexpr#1-1\relax}}
\makeatother

\newcommand{\N}{\mathbb{N}}

\newcommand{\R}{\mathbb{R}}
\newcommand{\C}{\mathbb{C}}

\newcommand{\e}{\varepsilon}

\newtheorem{theorem}{Theorem}
\newtheorem{lemma}[theorem]{Lemma}
\newtheorem{proposition}[theorem]{Proposition}

\newtheorem{corollary}[theorem]{Corollary}
\newtheorem{definition}[theorem]{Definition}
\newtheorem{claim}[theorem]{Claim}
\newtheorem{conjecture}[theorem]{Conjecture}

\theoremstyle{definition}

\newtheorem{remark}[theorem]{Remark}

\makeatletter
\newtheorem*{rep@theorem}{\rep@title}
\newcommand{\newreptheorem}[2]{%
\newenvironment{rep#1}[1]{%
 \def\rep@title{#2 \ref{##1}}%
 \begin{rep@theorem}}%
 {\end{rep@theorem}}}
\makeatother

\newreptheorem{theorem}{Theorem}
\newreptheorem{lemma}{Lemma}
\newreptheorem{proposition}{Proposition}

\PassOptionsToPackage{hyphens}{url}\usepackage{hyperref} 
\hypersetup{
  colorlinks=true,
  citecolor=blue,
  linkcolor=blue, 
  urlcolor=cyan
}

\def\?{{\bf ****???????****}}
\def\M{{\mathcal M}}

\newcommand{\defeq}{\vcentcolon=}

\long\def\frame#1#2#3#4{\hbox{\vbox{\hrule height#1pt
 \hbox{\vrule width#1pt\kern #2pt
 \vbox{\kern #2pt
 \vbox{\hsize #3\noindent #4}
\kern#2pt}
 \kern#2pt\vrule width #1pt}
 \hrule height0pt depth#1pt}}}

\newcommand\blfootnote[1]{%
	\begingroup
	\renewcommand\thefootnote{}\footnote{#1}%
	\addtocounter{footnote}{-1}%
	\endgroup
}

\title{The Nazarov proof of the non-symmetric\\ Bourgain--Milman inequality}
\author{Vlassis Mastrantonis, Yanir A. Rubinstein}
\date{12 June 2022}

\begin{document}

\let\i\undefined
\newcommand{\i}{\sqrt{-1}}

\maketitle

\begin{abstract}
In 2012, Nazarov used Bergman kernels and H\"ormander's $L^2$ estimates for the $\bar\partial$-equation
to give a new proof of the Bourgain--Milman theorem for symmetric convex
bodies and made some suggestions on how his proof should extend to general convex bodies. 
This article achieves this extension and serves simultaneously
as an exposition to Nazarov's work.
A key new ingredient is an affine invariant associated to
the Bergman kernel of a tube domain.
This gives the first `complex' proof of the Bourgain--Milman theorem for general convex 
bodies, specifically, without using symmetrization.
\end{abstract}

\blfootnote{Research supported by NSF grants 
DMS-1906370,2204347, BSF grants 2016173,2020329,
and a Hauptman Summer Research Award at the University of Maryland.
Y.A.R. thanks B. Berndtsson and B. Klartag for helpful discussions.}
\tableofcontents

\section{Introduction and main result}
Let $K\subset \R^n$ be a \textit{convex body}, that is a compact convex set with non-empty interior.  Its \textit{polar}
    $
    K^\circ\defeq \{y\in \R^n: \langle x,y\rangle\leq 1, \text{ for all } x\in K\}
    $
is convex as it is the intersection of convex sets (half-spaces); it is compact 
if and only if 
\begin{equation}
\label{intKeq}
0\in\mathrm{int}\,K,
\end{equation}
an assumption that we will use 
very often.
Also, $K$ is called symmetric
if \hfill\break
$K=-K\defeq
\{x\in\R^n: -x\in K\}.$

Let $A\in GL(n,\R)$. 
While $AK$ and $K$ could have wildly differing volume
(with respect to the Lesbegue measure $d\lambda$), the volume of the product body
$AK\times (AK)^\circ\subset \R^n\times\R^n$ is equal
to that of $K\times K^\circ$.
This leads to the following $GL(n,\R)$-invariant functional on convex bodies \cite[p. 95]{mahler}.
\begin{definition}
\label{Mahlerdef}
Let $K\subset \R^n$ be a convex body
satisfying \eqref{intKeq}. The \textit{Mahler volume} of $K$ is
\begin{equation*}
    \M(K)\defeq n! |K\times K^\circ|= n! |K||K^\circ|.
\end{equation*}
\end{definition}

Crude bounds on $\M$ were demonstrated by Mahler already in 1939 \cite[(6)]{mahler}.
In 1987,
Bourgain--Milman showed that there exists an unspecified but uniform $c>0$ independent of $K$ and $n$ such that 
\cite[Corollary 6.1]{bourgain-milman} 
\begin{equation}\label{bm_int}
    \M(K)\geq c^n,
\end{equation}
that---aside from determining the best value of $c$---is optimal in an
asymptotic sense \cite[pp. 149--150]{ryabogin-zvavitch}. 
One of Mahler's conjectures asserts that $c$ should be $4$ if $K$ is also symmetric \cite[p. 96]{mahler},
and another that for general convex $K$ \cite[(1)]{mahler2} \cite[p. 564]{schneider},
\begin{equation}\label{bm_int_ns}
    \M(K)\geq \frac{(n+1)^{n+1}}{n!}
\end{equation}
(this is, by Stirling's formula, asymptotic to $e^n$).

In 2012, Nazarov proved (\ref{bm_int}) 
with $c=\pi^3/16\approx 1.937$ for {\it symmetric $K$}. His proof
was pioneering in that he used complex methods,
namely, Bergman kernels and H\"ormander's $\overline{\partial}$-theorem. Moreover, he made some insightful suggestions \cite[p. 342]{nazarov} on how his
mainly {\it complex} arguments should extend to yield  (\ref{bm_int})
{\it for all $K$}
with $c=\pi/4\approx 0.785$
{\it without} using symmetrization techniques (but
with a sacrifice in the value of the constant, 
see Remark \ref{NonOptRk}). 
The purpose of the present article is to give an exposition of his
beautiful ideas and carry out the necessary computations mostly
following his suggestions. The main result of this article confirms
Nazarov's expected bound, and thus gives
the first proof of the Bourgain--Milman theorem for general convex 
bodies without using symmetrization:

\begin{theorem}\label{nazarov_lower_bound}
For all convex bodies $K\subset\R^n$, one has (\ref{bm_int}) with 
 $c= \pi/4$.
\end{theorem}

There are several points in the proof of Theorem \ref{nazarov_lower_bound}
that differ from Nazarov's analysis in the symmetric case:

\smallskip
\noindent 
{\it The weight function for H\"ormander's $L^2$-estimates.}
An interesting issue not present in the symmetric case comes from applying John's theorem to non-symmetric bodies. 
This comes up when one constructs the appropriate weight function to be used in H\"ormander's $L^2$-estimates. 
Yet without the assumption of symmetry, John's theorem does not guarantee that the maximal ellipsoid contained in the convex body will be centered at the origin. It may even happen that the origin is not at all contained in the maximal ellipsoid. In this latter case,
 $   B_2^n(a,r)\subset K\subset B_2^n(a,nr), $
and $0\notin B_2^n(a,r)$, i.e., $K^\circ\subset B_2^n(a,r)^\circ$ which is not a 
useful bound, since in this case $|B_2^n(a,r)^\circ|=\infty$. This does not allow for good control on several estimates, e.g., items (iii), (iv) and (v) in Lemma \ref{properties_list}. However, using Santal\'o's inequality one can overcome this by proving the estimate $K^\circ\subset B_2^n(0,\frac{2n}{r})$
(Lemma \ref{naz_john_lemma}). It turns out that this argument is not essential
for our goal, but we include it since it might be of independent interest;
we also give an alternative analytical argument that avoids the geometric arguments involving
John's theorem altogether and relies purely on tensorization
(Remark \ref{avoidJohnRk}).

\smallskip
\noindent 
{\it The bump function.}
In Nazarov's approach it is crucial to carefully construct 
a bump function to be used in setting up a $\bar\partial$-equation whose solution will allow to construct
a holomorphic $L^2$ function on a tube domain with good estimates
via H\"ormander's theorem. This largely follows Nazarov's ideas
in the symmetric case but the calculations are more involved in the general case.
For instance, a difference from the symmetric case is accounting for small perturbations of the center of some bodies, which are no longer centered at the origin. For example, the bump function constructed for the proof of Theorem \ref{nazarov_lower_bound} is no longer supported on $\sigma\delta K_\C\setminus \delta K_\C$, but rather, in $(\sigma\delta K_\C- (\sigma-1)\delta (a+\sqrt{-1}a))\setminus \delta K_\C$ (see Lemma \ref{controlled_bump_function}).

\smallskip
\noindent 
{\it An affine invariant.}
The proof of Proposition \ref{nazprop2}
relies on both the weight function and the bump function discussed above.
However, once again, because of the possibly awkward position of the non-symmetric body after applying John's theorem it is necessary to
find a quantity that is controlled. Fortunately, 
we observe the affine invariance of 
$\mathcal{B}(K):=|K|^2\mathcal{K}_{T_K}(\sqrt{-1}b(K), \sqrt{-1}b(K))$ (see Definition \ref{BKDef},
Lemma \ref{bk_ai}, and \S\ref{sec4.4})
that is new compared to the symmetric case and allows
to complete the proof of Proposition \ref{nazprop2},
and hence of Theorem~\ref{nazarov_lower_bound}.

Several remarks follow to place Theorem 
\ref{nazarov_lower_bound} in context.

\begin{remark}
\label{NonOptRk}
Applying Nazarov's complex methods directly to non-symmetric convex 
bodies---as we do in this article---gives (\ref{bm_int}) with $c=\pi/4\approx 0.785$. On the other hand, as pointed out by Nazarov,
replacing a given non-symmetric $K$ with an associated symmetrization of $K$
and then applying Nazarov's bound for symmetric bodies yields
a better estimate, namely, 
$\frac12 \pi^3/16\approx 0.968$ \cite[p. 342]{nazarov}.
This well-known ``symmetrization trick" is briefly described in Corollary \ref{sym_cor}
for the sake of exposition.
Thus, the main point of this article, as in Nazarov's original article, 
is not to derive the best possible constant but rather to highlight that
Nazarov's complex methods {\it give the first
proof of the Bourgain--Milman theorem for general bodies without using 
symmetrization}. For previous proofs that use symmetrization see, e.g.,
 \cite[Lemma 3.1]{bourgain-milman}, \cite[Theorem 1.4]{milman}, \cite[Corollary 1.6]{kuperberg2}, \cite[Theorem 1.3]{giannopoulos-paouris-vritsiou}. 
\end{remark}

\begin{remark}
B\l{}ocki recovered one of Nazarov's estimates on the Bergman kernel (Proposition \ref{nazprop2}) by providing lower bounds for the Bergman kernel of a 
convex domain in $\C^n$ \cite[Theorem 2]{blocki2}.
As we explain in a forthcoming publication  \cite{MR}, 
even though also B\l{}ocki only considered  symmetric convex bodies \cite[\S 4]{blocki}, his approach readily applies to the non-symmetric case, yielding another proof of Theorem \ref{nazarov_lower_bound}. However, we believe that the approach presented here is 
more accessible.
\end{remark}

\begin{remark}
The best known constant for (\ref{bm_int}) in dimensions $n\geq 4$ with $K$ symmetric is  $c=\pi$ \cite[Corollary 1.6]{kuperberg2}, \cite[Theorem 2.1]{berndtsson1}. The sharp bound $c=4$ is due to Mahler in dimension $n=2$ \cite[(2)]{mahler2}, and Iriyeh--Shibata in dimension $n=3$ \cite[Theorem 1.1]{iriyeh-shibata} (cf. Fradelizi et al. \cite{fradelizi-et_al}). For general $K$, the best known constant is $c=2$ for $n=3$ and $c=\pi/2$
for $n\geq 4$ by the symmetric bound and Corollary \ref{sym_cor}. In dimension $n=2$ the sharp bound \eqref{bm_int_ns} is due to Mahler \cite[(1)]{mahler2}.
\end{remark}

\noindent
{\bf Organization.}
In Section \ref{sec2} basic facts about Bergman kernels are given,
the functional $\mathcal{B}$ is introduced (Definition \ref{BKDef}),
and its affine invariance is stated (Lemma \ref{bk_ai}).
This is followed by the statement of two key estimates (Propositions \ref{nazprop1} and \ref{nazprop2}) on the Mahler volume $\M$ and on $\mathcal{B}$
involving Bergman kernels of tube domains 
that lead to the proof of Theorem \ref{nazarov_lower_bound}. 
The key ideas in the proof of Proposition \ref{nazprop1} 
are outlined at the beginning of
\S\ref{nazsub2} and the detailed proof occupies the remainder of that subsection.
It relies on the well-known
Paley--Wiener correspondence for tube domains that is carefully derived in 
\S\ref{sec3.1} relying and expanding on several sources
\cite[\S 3]{berndtsson2}, \cite[\S 3]{blocki}, \cite[\S 4]{hultgren}.
Section \ref{sec4} derives a lower bound on $\mathcal{B}$
and we refer to the beginning of that section for a detailed step-by-step road-map
for the proof. Many of the steps are adaptations
(some rather straightforward, some quite technical) of
Nazarov's arguments from the symmetric setting \cite[\S 5--6]{nazarov},
yet several steps are new to the non-symmetric setting and the study of $\mathcal{B}$.
First, the affine invariance of 
$\mathcal{B}(K)=|K|^2\mathcal{K}_{T_K}(\sqrt{-1}b(K), \sqrt{-1}b(K))$ 
is derived in \S\ref{sec4.4}.
Second, a convenient displacement of $K$ is studied in \S\ref{repositionSec}.
It is at this point in the analysis that the use of Santal\'o's theorem occurs (Lemma \ref{naz_john_lemma}), though
we also give an alternative argument that avoids this feature
(Remark \ref{avoidJohnRk}).
Third, Nazarov's plurisubharmonic support function is extended to the
non-symmetric setting in \S\ref{pshsupp_sub} leading to the
definition of the weight function.
Next, Lemma \ref{properties_list} in 
\S\ref{weight_function_sub} contains 
the properties needed from the weight function 
for the application of H\"ormander's $L^2$ estimates. 
In \S\ref{bump_function_sub}
a bump function is constructed to be used in setting up a $\bar\partial$-equation whose solution will allow to construct
a holomorphic $L^2$ function on the tube domains with good estimates. 
This follows Nazarov's ideas
in the symmetric case but the calculations are more involved in the general case.
The proof of Proposition \ref{nazprop2}
occupies \S\ref{lower_bound_sub} and relies on the ingredients above together
with a standard ``tensorization trick" for Bergman kernels
described in \S\ref{tensorSec}.
Finally, Appendix \ref{sec_sym} serves as an explanation (though
not self-contained) of the classical symmetrization trick that has been prevalent in other approaches to the Bourgain--Milman theorem  \cite[p. 342]{nazarov}.


\section{Mahler volume and Bergman kernel of tube domains}\label{sec2}

This section introduces the functional $\mathcal{B}$ 
(Definition \ref{BKDef}) and
provides the proof of Theorem \ref{nazarov_lower_bound},
modulo two key estimates (Propositions \ref{nazprop1}
and \ref{nazprop2}) and the affine invariance of $\mathcal{B}$
(Lemma \ref{bk_ai}).

\subsection{Bergman spaces}
\label{BergSec}

Nazarov's key idea is to reduce the proof of Theorem \ref{nazarov_lower_bound} to the study of the Hilbert space of $L^2$-integrable holomorphic functions
\begin{equation*}
    A^2(T_K)\defeq \{f: T_K\to \C: f \text{ holomorphic, } \|f\|^2_{L^2(T_K)}\defeq \int_{T_K}|f(z)|^2\dif\lambda(z)< \infty\},
\end{equation*}
on so-called `tube domains 
over $K$',
\begin{equation}
\label{TKEq}
T_K\defeq 
\R^n+ \sqrt{-1}(\mathrm{int}\,K)
\subset \C^n.
\end{equation}

Fix
a convex body $K\subset \R^n$. 
For $w\in T_K$, the evaluation map
\begin{gather*}
        \mathrm{ev}_{w}\equiv \mathrm{ev}_{T_K,w}: f\in A^2(T_K)\mapsto f(w) \in\C, 
\end{gather*}
is a bounded linear functional 
from $A^2(T_K)$ equipped with
the $L^2(T_K)$ norm to $(\C, |\,\cdot\,|)$:
this follows
from the holomorphicity of $f$, which implies that $|f|^2$ is subharmonic,
so for $\e>0$ such that $B_2^{2n}(w,\e)\subset T_K$, by
the mean value inequality, 
\begin{equation*}
    |\mathrm{ev}_{T_K,w}(f)|^2= |f(w)|^2\leq \frac{1}{|B_2^{2n}(w,\e)|}\int_{B_2^{2n}(w,\e)}|f(z)|^2\dif\lambda(z)\leq \frac{\|f\|^2_{L^2(T_K)}}{\e^{2n}|B_2^{2n}(0,1)|},
\end{equation*}
so 
\begin{equation*}
||\mathrm{ev}_{T_K,w}||_{(A^2(T_K),L^2(T_K)),(\C, |\,\cdot\,|)}\defeq
     \sup_{\substack{0\neq f\in A^2(T_K)}} \frac{|f(w)|}{\|f\|_{L^2(T_K)}}
     \le {\e^{-n}|B_2^{2n}(0,1)|^{-1/2}},
\end{equation*}
that is bounded as claimed.

Thus, the Riesz representation theorem 
provides
\begin{equation}\label{bergmanL2}
\mathcal{K}_{T_K}(\,\cdot\,,w)\in A^2(T_K),
\end{equation}
satisfying
\begin{equation}
\label{reproducingEq}
    f(w)= \langle f, \mathcal{K}_{T_K}(\cdot, w)\rangle_{L^2(T_K)}= \int_{T_K}f(z)\overline{\mathcal{K}_{T_K}(z,w)}\dif\lambda(z),
\end{equation}
for all $f\in A^2(T_K)$.
The reproducing kernel $\mathcal{K}_{T_K}$, considered as a function on 
$T_K\times T_K$, 
is the \textit{Bergman kernel} of the tube domain.
By (\ref{bergmanL2}) 
$\mathcal{K}_{T_K}$ is holomorphic in $z$.
Since in \eqref{reproducingEq} $f\in A^2(T_K)$, i.e., $f$ is 
holomorphic, $\mathcal{K}_{T_K}$ 
is also anti-holomorphic in $w$.

\subsection{Two estimates on the Bergman kernel}

In essence, there are two key estimates in the proof of Theorem \ref{nazarov_lower_bound},
following a standard preliminary step 
(described in detail in the proof below)
involving a translation by the barycenter
\begin{equation}
\label{bKEq}
    b(K)\defeq \frac{1}{|K|}\int_K x\dif x. 
\end{equation}

\noindent
{\it The first estimate.} The first step is a lower bound on the Mahler volume of 
the translated body in terms of the Bergman kernel.

\begin{proposition}\label{nazprop1}
For a convex body $K\subset \R^n$ and $a\in \mathrm{int}\,K$,
\begin{equation*}
\M(K-a)\geq \pi^n |K|^2 \mathcal{K}_{T_K}(\sqrt{-1}a,\sqrt{-1}a).
\end{equation*}
\end{proposition}

Proposition \ref{nazprop1} leads to the following functional
on convex bodies:

\begin{definition}
\label{BKDef}
For a convex body $K\subset \R^n$ let 
\begin{equation*}
    \mathcal{B}(K)\defeq |K|^2 \mathcal{K}_{T_K}(\sqrt{-1}b(K), \sqrt{-1}b(K)).
\end{equation*}
\end{definition}

An elementary new observation that is crucial for this article is:
\begin{lemma}\label{bk_ai}
$\mathcal{B}(K)$ is an affine invariant.
\end{lemma}

Proposition \ref{nazprop1} will be used 
with $a=b(K)$ \eqref{bKEq} since then
both sides of 
\begin{equation}\label{prop1eq}
    \M(K-b(K))\geq \pi^n \mathcal{B}(K)
\end{equation}
are affine invariants and, moreover, the right hand
side will be shown to have a uniform lower bound 
(Propostion \ref{nazprop2}). 
The affine invariance of the right-hand side 
is the content of Lemma \ref{bk_ai}.
To see the affine invariance 
of the left-hand side is more straightforward:
let $S(x)= Ax+ b$ for $A\in GL(n,\R), b\in \R^n$, be an affine transformation. Since $b(S(K))= S(b(K))$ (\ref{affine_barycenter}),
\begin{equation*}
    \begin{aligned}
        \M(S(K)- b(S(K)))&= \M(S(K)-S(b(K)))= \M(AK+b- Ab(K)-b)\\ &= \M(A(K-b(K)))= \M(K-b(K)),
    \end{aligned}
\end{equation*}
because $\M$ is invariant under the action of $GL(n,\R)$ and 
\begin{equation}\label{affine_barycenter}
    b(S(K))= \frac{1}{|AK+b|}\int_{AK+b} x\dif x= \frac{1}{|\det A||K|}\int_K (Ay+b) |\det A|\dif y= A b(K)+b= S(b(K)),
\end{equation}
since $A$ is linear, and hence commutes with the integral. 

Nazarov proves a special case of Proposition \ref{nazprop1} for symmetric convex bodies \cite[p. 338]{nazarov}: 
\begin{equation}\label{Nazrov-Prop6}
\M(K)\geq |K|^2 \mathcal{K}_{T_K}(0,0) \pi^n.
\end{equation}
We extend his estimate to general convex bodies 
by observing that his proof does not actually
require symmetry (which implies $b(K)=0$)
(see, e.g., Lemma \ref{J_K-lower_bound}).
A similar observation was already made by Hultgren
who derived a functional version of \eqref{Nazrov-Prop6} 
for convex functions $f$ with $\int_{\R^n} xe^{-f(x)}\dif x=0$
\cite[Lemma 11]{hultgren}. 
The proof of Proposition \ref{nazprop1} occupies \S\ref{nazsub2}.
It relies on the well-known
Paley--Wiener correspondence for tube domains that is carefully derived in \S\ref{sec3.1}. 
Lemma \ref{bk_ai} is not part of the 
proof of Proposition \ref{nazprop1} and is proved in \S\ref{sec4.4}.

\medskip
\noindent
{\it The second estimate.}
So far there is little distinction between symmetric and non-symmetric convex bodies.
The essential differentiation between the two cases comes in 
the next estimate concerning a lower bound on the Bergman kernel
on the diagonal:

\begin{proposition}\label{nazprop2}
For $K\subset \R^n$ a convex body,
\begin{equation*}
 \mathcal{B}(K)\geq 4^{-n}. 
\end{equation*}
\end{proposition}

\begin{conjecture}
For $K\subset \R^n$ a convex body, 
$ \mathcal{B}(K) \ge  \mathcal{B}(\Delta_n)$ 
where\hfill\break $\Delta_n\defeq \{x\in [0, \infty)^n: x_1+ \ldots+ x_n\leq 1\}$ is the $n$-dimensional simplex.
\end{conjecture}

\begin{remark}
Nazarov's analogue of Proposition \ref{nazprop2}
for symmetric bodies is
\begin{equation}\label{naz_KC}
     \mathcal{K}_{T_K}(0,0)\geq \left( \frac{\pi^2}{16}\right)^n/|K_\C|, 
\end{equation}
where $|K_\C|\leq |K|^2$ as in (\ref{KC}) \cite[p. 341]{nazarov}. 
While \eqref{naz_KC} is sharp \cite[p. 342]{nazarov},
if one were to replace $|K_\C|$ by $|K|^2$ it would no longer be.
For us, Proposition \ref{nazprop2} is not sharp and would not
be sharp even if  $|K|^2$ were replaced by $|K_\C|$
(which is possible by Lemma \ref{g_estimate}). Indeed, in dimension $n=1$,  $K= (-1/2, 1/2)$, $\mathcal{K}_{T_K}(0,0)= \pi/4$ and $|K|=1$ \cite[p. 342]{nazarov}.
Thus, by affine invariance (Lemma \ref{bk_ai})
\begin{equation*}
    |K|^2\mathcal{K}_{T_K}(0,0)= \frac{\pi}{4}> \frac14, 
\end{equation*}
for all symmetric intervals $K$ in $\R$ (i.e., $K\subset \R$ with $b(K)=0$). Moreover,  if $\mathcal{B}(K)$ were to be replaced by $|K_\C|\mathcal{K}_{T_K}(\sqrt{-1}b(K), \sqrt{-1}b(K))$ in Proposition \ref{nazprop2}, the estimate would still not be sharp since in dimension $n=1$ the estimate (\ref{naz_KC}) is sharp \cite[p. 342]{nazarov}. 
\end{remark}
\begin{remark}
B\l{}ocki conjectured that for symmetric convex bodies
$
    \mathcal{K}_{T_K}(0,0)\geq \left( \frac{\pi}{4}\right)^n/|K|^2,
$
attained for the cube $[-1,1]^n$ (that would imply 
\eqref{bm_int} with $c=\pi^2/4$
for symmetric convex bodies) \cite[(7)]{blocki2}. 
\end{remark}

\begin{proof}[Proof of Theorem \ref{nazarov_lower_bound}]
For convex bodies $K\subset \R^n$ with $b(K)=0$, the claim follows from Propositions \ref{nazprop1} and \ref{nazprop2}. In general, for any convex body $K\subset \R^n$ the volume product
\begin{equation}\label{naz_eq12}
    \inf_{z\in \R^n}\M(K-z)= \M(K-s(K))
\end{equation}
is {\it minimized} at a unique point $s(K)\in \mathrm{int}\,K$ (called the Santal\'o point) for which $b((K-s(K))^\circ)=0$ \cite[(2.3)]{santalo}. 
The Mahler volume of the translated body $K-s(K)$
equals that  of its polar $(K-s(K))^\circ$, and the
latter is bounded from below by $(\pi/4)^n$ as its barycenter is at the origin.
\end{proof}

\section{Estimating the Mahler volume}\label{sec3}
\label{nazsub1}

This section culminates in \S\ref{nazsub2} in the proof of Proposition \ref{nazprop1}
that states a lower bound for the Mahler volume in terms of a Bergman kernel.
Since the former can be expressed as an integral involving the support function
$h_K$ \eqref{h_K-def} (Claim \ref{integral_polar_support}),
the gist of the proof is to recognize that the support function 
has an ``$L^1$-cousin" in the form of a functional $\tilde{h}_K$ (Definition \ref{tildehKDef}), 
that this cousin actually can bound its ``$L^\infty$-cousin" $h_K$
(Lemma \ref{J_K-lower_bound}), and that this $\tilde{h}_K$, 
in turn,
is closely related to the Bergman kernel of the tube domain $T_K$ over $K$.

\subsection{A Paley-Wiener correspondence}\label{sec3.1}

The main result of this subsection is 
the Paley--Wiener correspondence, Theorem \ref{pw-thm},
and we mainly follow  Berndtsson \cite[Proposition 3.1]{berndtsson2}, B\l{}ocki \cite[Section 3]{blocki}, Hultgren \cite[Chapter 4]{hultgren}, and Nazarov \cite[Section 3]{nazarov},
but add more detail as needed.

\subsubsection{The Paley--Wiener map}

The Mahler volume naturally involves the support function 
\begin{equation}\label{h_K-def}
    h_K(y)\defeq \sup_{x\in K}\langle x,y\rangle. 
\end{equation}
The key idea relating $\M(K)$ to $\mathcal{K}_{T_K}$ is an $L^1$ analogue.

\begin{definition}
\label{tildehKDef}
For $K\subset\R^n$ a compact body, let
\begin{equation*}
    \tilde h_K(x)\defeq \log\frac1{|K|}\int_{K} e^{\langle x,y\rangle}\dif y,
\end{equation*}
and denote by $L^2(\R^n, \tilde h_K)$ the class of functions $g:\R^n\to \R$ such that
\begin{equation*}
    \|g\|_{L^2(\tilde{h}_K)}\defeq \left(|K|\int_{\R^n}|g(x)|^2 
    e^{\tilde h_K(-2x)}
    \dif x\right)^\frac12<\infty. 
\end{equation*}
\end{definition}

By compactness of $K$, $h_K\ge \tilde h_K$. 
A key observation is that convexity yields a reverse inequality
(Lemma \ref{J_K-lower_bound}). 
Proposition \ref{nazprop1} then readily follows
since the Bergman kernel is closely
related to $\tilde h_K$ by a classical formula 
(that can be justified by Theorem \ref{pw-thm}).

\begin{remark}
Nazarov \cite[p. 337]{nazarov} (and B\l{}ocki \cite[p. 93]{blocki})
define
$
    J_K(x)\defeq \int_{K} e^{-2\langle x,y\rangle}\dif y, 
$
and work with the class $L^2(\R^n, J_K)$ of weighted $L^2$-integrable functions $g: \R^n\to \R$ such that
$    \|g\|_{L^2(J_K)}\defeq \left( \int_{\R^n} |g(x)|^2 J_K(x)\dif x \right)<\infty. 
$
Since $J_K(x)= |K|e^{\tilde{h}_K(-2x)}$, $L^2(\R^n, J_K)= L^2(\R^n, \tilde{h}_K)$. 
Working with $\tilde h_K$ makes some of the key estimates more 
intuitive geometrically (e.g., Lemma \ref{J_K-lower_bound}). 

\end{remark}

The following Paley--Wiener type theorem establishes an integral representation of the elements of $A^2(T_K)$ in terms
of elements of $L^2(\R^n, \tilde{h}_K)$. Define a map $\mathrm{PW}$ sending functions in $L^2(\R^n, \tilde{h}_K)$ to functions on $T_K$,
\begin{equation}\label{pw_def}
    \begin{gathered}
        L^2(\R^n, \tilde{h}_K)\ni g\mapsto 
        \mathrm{PW}(g)(w)\defeq\frac{1}{(2\pi)^\frac{n}{2}}\int_{\R^n} g(x)e^{\sqrt{-1}\langle w,x\rangle}\dif x, \quad w\in T_K.
    \end{gathered}
\end{equation}

\begin{theorem}\label{pw-thm}
$\mathrm{PW}$
is a bijection between $L^2(\R^n, \tilde{h}_K)$ and $A^2(T_K)$, with
\hfill\break $\|\mathrm{PW}(\,\cdot\,)\|_{L^2(T_K)}=  \|\,\cdot\,\|_{L^2(\tilde{h}_K)}$.
\end{theorem}

Theorem \ref{pw-thm} 
establishes not just an integral
representation, but one that
is also an isometry between
the respective Hilbert space
structures.
Theorem \ref{pw-thm} 
is well-documented, see, e.g., \cite[\S 3]{berndtsson2}, \cite[\S 3]{blocki}, \cite[\S 4]{hultgren},
but the details are scattered and often left to the reader
so we provide a streamlined proof for the reader's convenience.
Theorem \ref{pw-thm} follows from the following four propositions. 

\begin{proposition}\label{pw-prop1}
For $g\in L^2(\R^n, \tilde{h}_K)$ and $w\in T_K$ the integral $\int_{\R^n} g(x)e^{\sqrt{-1}\langle w,x\rangle}\dif x$ converges in $\C$.
\end{proposition}

\begin{proposition}\label{pw-prop2}
$\|\mathrm{PW}(\cdot)\|_{L^2(T_K)}= \|\cdot\|_{L^2(\tilde{h}_K)}$.
\end{proposition}

\begin{proposition}\label{pw-prop3}
$\mathrm{PW}(L^2(\R^n, \tilde{h}_K))\subset A^2(T_K)$.
\end{proposition}

\begin{proposition}\label{pw-prop4}
$A^2(T_K)\subset \mathrm{PW}(L^2(\R^n, \tilde{h}_K))$.
\end{proposition}

\begin{proof}[Proof of Theorem \ref{pw-thm}]
By Proposition \ref{pw-prop1}, for any $g\in L^2(\R^n, \tilde{h}_K)$, $\mathrm{PW}(g)$ is a $\C$-valued function on $T_K$. By Proposition \ref{pw-prop2}, $\mathrm{PW}$ is injective since it is linear and $\mathrm{PW}(g)=0$ if and only if $\|\mathrm{PW}(g)\|_{L^2(T_K)}= \|g\|_{L^2(\tilde{h}_K)}=0$, i.e, $g=0$. Surjectivity follows from Propositions \ref{pw-prop3} and \ref{pw-prop4}.
\end{proof}

\subsubsection{Convergence of the integral
}\label{nazsubsub11}
For the proof of Propositions \ref{pw-prop1} we follow B\l{}ocki \cite[p. 93]{blocki}. Essentially, the function $\mathrm{PW}(g)$ can be estimated by putting absolute value inside the integrand. This, naturally, leads to the appearance of $\tilde{h}_K$, which itself can be estimated from below (Lemma \ref{sinh_lemma}).

\begin{proof}[Proof of Proposition \ref{pw-prop1}]
Let $w= \xi +\sqrt{-1}v\in T_K$, that is $\xi\in \R^n, v\in \mathrm{int}\,K$. 
By Cauchy--Schwarz,
\begin{equation*}
\begin{aligned}
    |\mathrm{PW}(g)(w)|&= \left|\int_{\R^n} g(x)e^{\sqrt{-1}\langle x,w\rangle}\dif x\right| \leq \int_{\R^n}|g(x)|e^{-\langle x,v\rangle}\dif x \\
    &= \int_{\R^n} |g(x)|\sqrt{|K|e^{\tilde{h}_K(-2x)}}\frac{e^{-\langle x,v\rangle}}{\sqrt{|K|e^{\tilde{h}_K(-2x)}}}\dif x \\
    &\leq \left(|K|\int_{\R^n} |g(x)|^2 e^{\tilde{h}_K(-2x)}\dif x\right)^\frac12 \left(\frac{1}{|K|}\int_{\R^n} e^{-2\langle x,v\rangle-\tilde{h}_K(-2x)} \dif x \right)^\frac12
\end{aligned}
\end{equation*}
To estimate the last term, there exists some $r>0$ such that $v+ [-r,r]^n\subset \mathrm{int}\,K$. As a result,
\begin{equation}\label{1.1nazeq2}
    |K|e^{\tilde{h}_K(-2x)}\geq \int_{v+ [-r,r]^n} e^{-2\langle x,y\rangle}\dif y= \prod_{i=1}^n\int_{v_i-r}^{v_i+r} e^{-2x_iy_i}\dif y_i= e^{-2\langle x,v\rangle}\prod_{i=1}^n \frac{\sinh(2rx_i)}{x_i}.
\end{equation}
By Lemma \ref{sinh_lemma},
\begin{equation}\label{nazeq1}
    \frac{1}{|K|}\int_{\R^n} e^{-2\langle x,v\rangle- \tilde{h}_K(-2x)}\dif x\leq \left( \int_\R \frac{s}{\sinh(2rs)}\dif s\right)^n= \left( \frac{\pi^2}{8r^2}\right)^n.
\end{equation}
Since, $g\in L^2(\R^n, \tilde{h}_K)$, 
$\mathrm{PW}(g)(w)\in \C$ for each $w\in T_K$, proving Proposition \ref{pw-prop1}.
\end{proof}

The following was used for the integral of ${s}/{\sinh(2rs)}$.
\begin{lemma}\label{sinh_lemma}
For $r>0$,
$
    \displaystyle    \int_{\R} 
\frac{t\dif t}{\sinh(2rt)}= \frac{\pi^2}{8r^2}.
$
\end{lemma}
\begin{proof}
Expand the integrand, 
\begin{equation}\label{sinheq1}
    \frac{t}{\sinh(2rt)}= \frac{2t}{e^{2rt}- e^{-2rt}}= \frac{2t e^{-2rt}}{1-(e^{-2rt})^2}= \sum_{k=0}^\infty 2t e^{-2rt(k+1)}.
\end{equation}
Using integration by parts,
\begin{equation}\label{sinheq2}
    \int_{0}^{\infty} 2te^{-2rt(k+1)}\dif t= \frac{1}{r(k+1)}\int_0^{\infty} e^{-2rt(k+1)}\dif t= \frac{1}{2r^2(k+1)^2}.
\end{equation}
By (\ref{sinheq1}), (\ref{sinheq2}), and Tonelli's theorem \cite[\S 2.37]{folland} (see Claim \ref{fubini-tonelli} below), since the integrand is an even function,
\begin{equation*}
\begin{aligned}
    \int_{\R}\frac{t}{\sinh(2rt)}\dif t&= 2\int_0^{\infty}\frac{t}{\sinh(2rt)}\dif t= 2\sum_{k=0}^{\infty}\int_{0}^{\infty}2t e^{-2rt(k+1)}\dif t\\
    &= 2\sum_{k=0}^\infty \frac{1}{2r^2(k+1)^2}=\frac{1}{r^2} \frac{3}{4}\sum_{k=1}^\infty \frac{1}{k^2} =\frac{3}{4r^2}\frac{\pi^2}{6}= \frac{\pi^2}{8r^2}.
\end{aligned}
\end{equation*}
\end{proof}
\begin{remark}\label{l2condition}
Similarly, an $L^2$ property for ${e^{-2\langle \,\cdot\,,v\rangle- \tilde{h}_K(-2(\,\cdot\,))}}$ can derived
(this will be useful in proving the formula for the Bergman
kernel of a tube domain (Lemma \ref{bk-tube})). 
By (\ref{sinheq1}), 
\begin{equation*}
    \frac{t^2}{\sinh(2rt)^2}= \frac{4t^2 e^{-4rt}}{\left( 1- (e^{-2rt})\right)^2}= \sum_{k=1}^\infty 4t^2 k e^{-4rtk}, 
\end{equation*}
thus, 
\begin{equation}\label{sinheq3}
    \int_{\R}\frac{t^2}{\sinh(2rt)^2}\dif t= 2\int_{0}^\infty \frac{t^2}{\sinh(2rt)^2}\dif t= 2\sum_{k=1}^\infty \int_0^\infty 4t^2=  e^{-4rtk}\dif t= \frac{1}{8r^3}\sum_{k=1}^\infty \frac{1}{k^3}
\end{equation}
which is finite. As a result, by (\ref{1.1nazeq2}) and (\ref{sinheq3}), 
$\displaystyle    \int_{\R^n} \left(e^{-2\langle x,v\rangle-\tilde{h}_K(-2x)}\right)^2\dif x 
$
is also finite. 
\end{remark}

\subsubsection{Fourier transform and integration tools}

This subsection recalls some  elementary real analysis tools that will be used repeatedly throughout.

For $\xi\in\R^n$,
\begin{equation}\label{fourier_transform}
    \hat{g}(\xi)\defeq \int_{\R^n} g(x) e^{-\sqrt{-1}\langle \xi,x\rangle}\dif x,
\end{equation}
is the Fourier transform of $g$. Strictly speaking (\ref{fourier_transform}) requires $g\in L^1(\R^n)$, but relaxing (\ref{fourier_transform}) to hold a.e. one can allow $g\in L^2(\R^n)$ \cite[Theorem 7.1.11]{hormander}. For $g\in L^2(\R^n)$, by Fourier inversion \cite[(7.1.4)]{hormander},
\begin{equation}\label{fourier_inv}
    g(x)= (2\pi)^{-n}\int_{\R^n}\hat{g}(\xi)e^{\sqrt{-1}\langle \xi,x\rangle}\dif \xi.
\end{equation}
Combining (\ref{fourier_transform}) and (\ref{fourier_inv}), and flipping the sign of $\xi$, 
\begin{equation}\label{fourier_inv2}
    g(x)= (2\pi)^{-n}\int_{\R^n}\int_{\R^n}g(s) e^{\sqrt{-1}\langle \xi,x-s\rangle}\dif s\dif \xi= (2\pi)^{-n}\int_{\R^n}\int_{\R^n} g(s)e^{\sqrt{-1}\langle \xi, s-x\rangle}\dif s\dif \xi.
\end{equation}
Moreover, for $f\in L^2(\R^n)$, 
\begin{equation}\label{parseval}
    \int_{\R^n}f(x)\overline{g(x)}\dif x= \frac{1}{(2\pi)^n}\int_{\R^n}\hat{f}(\xi)\overline{\hat{g}(\xi)}\dif \xi, 
\end{equation}
\cite[Theorem 7.1.6]{hormander} and, in particular, 
\begin{equation}\label{planch_eq}
    \|\hat{g}\|_{L^2(\R^n)}= (2\pi)^\frac{n}{2}\|g\|_{L^2(\R^n)}, 
\end{equation}
the so called Plancherel's theorem \cite[\S 8.29]{folland}.

Recall the theorems attributed to Tonelli and Fubini  \cite[\S 2.37]{folland}.

\begin{claim}\label{fubini-tonelli}
For $n,m\in\mathbb{N}$, denote by $v= (x,y)\in \R^n\times \R^m$.

\noindent
(i) For non-negative measurable $f:\R^n\times \R^m\to [0, \infty)$, 
\begin{equation}\label{tonelli}
\begin{aligned}
    \int_{\R^{n+m}} f(v)\dif\lambda_{n+m}(v)&= \int_{\R^n}\left( \int_{\R^m} f(x,y)\dif\lambda_m(y)\right) \dif\lambda_n(x) \\
    &= \int_{\R^m}\left( \int_{\R^n} f(x,y)\dif\lambda_n(x)\right)\dif\lambda_m(y).
\end{aligned}
\end{equation}

\noindent 
(ii) For $f\in L^1(\R^n\times \R^m)$, $x\mapsto f(x,y)$ is $L^1$-integrable for almost all $y\in \R^m$, and $y\mapsto f(x,y)$ is $L^1$-integrable for almost all $x\in \R^n$ with

\begin{align}\label{f1}
    \int_{\R^{n+m}} f(v)\dif\lambda_{n+m}(v)&= \int_{\R^n}\left(\int_{\R^m}f(x,y)\dif \lambda_m(y)\right)\dif \lambda_n(x)\\\label{f2}
    &= \int_{\R^m}\left( \int_{\R^n}f(x,y)\dif\lambda_n(x)\right)\dif \lambda_m(y).
\end{align}

\end{claim}

\begin{remark}
Since
part (i) (Tonelli's theorem)
does not assume integrability of $f$, it is often used to justify part (ii)
(Fubini's theorem): given measurable $f: \R^n\times \R^m\to \R$ one may compute any of the two iterated integrals 
\begin{equation*}
    \int_{\R^n}\left(\int_{\R^m} |f(x,y)|\dif y\right)\dif x \quad \text{ or } \quad \int_{\R^m}\left(\int_{\R^n}|f(x,y)|\dif x \right)\dif y;
\end{equation*}
if either is finite, by (\ref{tonelli}), $f\in L^1(\R^n\times \R^m)$ justifying the use of Fubini's theorem (Claim \ref{fubini-tonelli}~(ii)).
\end{remark}

\subsubsection{\texorpdfstring{$\mathrm{PW}$}{TEXT} is an isometry between \texorpdfstring{$L^2$}{TEXT} spaces}

Proposition \ref{pw-prop1} shows that for any $g\in L^2(\R^n, \tilde{h}_K)$, $\mathrm{PW}(g)$ is a $\C$-valued function on $T_K$. The following shows it is also $L^2$-integrable and has $L^2(T_K)$ norm equal to $\|g\|_{L^2(\tilde{h}_K)}$
\cite[(6)]{blocki} \cite[Proposition 3.1]{berndtsson2}.

\begin{proof}[Proof of Proposition \ref{pw-prop2}]
For $w=\xi+\sqrt{-1}v\in T_K$,
\begin{equation*}
    \mathrm{PW}(g)(w)= \frac{1}{(2\pi)^\frac{n}{2}}\int_{\R^n}g(x)e^{\sqrt{-1}\langle w,x\rangle}\dif x= \frac{1}{(2\pi)^\frac{n}{2}}\int_{\R^n} g(x)e^{-\langle x,v\rangle} e^{\sqrt{-1}\langle x,\xi\rangle}\dif x. 
\end{equation*}
Since
by Lemma \ref{tgtL2Lemma}
below
$x\mapsto (2\pi)^{n/2}g(x)e^{-\langle x,v\rangle}$ is $L^2$-integrable, by (\ref{fourier_transform}) and (\ref{fourier_inv}) $x\mapsto (2\pi)^{n/2} g(x)e^{-\langle x,v\rangle}$ is the Fourier transform of $\xi\mapsto \mathrm{PW}(g)(\xi+\sqrt{-1}v)$. 
Moreover, by (\ref{planch_eq})
\begin{equation}\label{nazeq1.1}
    (2\pi)^n \int_{\R^n}|g(x)|^2 e^{-2\langle x,v\rangle}\dif x= (2\pi)^n \int_{\R^n} |\mathrm{PW}(g)(\xi+\sqrt{-1}v)|^2\dif \xi.
\end{equation}
Integrating (\ref{nazeq1.1}) with respect to $v\in K$
and interchanging the order of integration (by (\ref{tonelli})), 
\begin{equation*}
    \|\mathrm{PW}(g)\|_{L^2(T_K)}^2= \int_{\mathrm{int}\,K}\int_{\R^n} |g(x)|^2 e^{-2\langle x,v\rangle}\dif x\dif v= |K|\int_{\R^n}|g(x)|^2 e^{\tilde{h}_K(-2x)}\dif x =\|g\|_{L^2(\tilde{h}_K)}^2, 
\end{equation*}
by Definition \ref{tildehKDef}.
\end{proof}

\begin{lemma}
\label{tgtL2Lemma}
For $g\in L^2(\R^n, \tilde{h}_K)$ and $v\in \mathrm{int}\,K$,
$x\mapsto g(x)e^{-\langle x,v\rangle}$ is in $L^2(\R^n)$. 
\end{lemma}
\begin{proof}
As $|s|\leq 2|\sinh s|$ for all $s\in\R$, by 
(\ref{1.1nazeq2}) there exists $r>0$ satisfying
\begin{equation*}
    |K|e^{\tilde{h}_K(-2x)}\geq e^{-2\langle x,v\rangle}\prod_{i=1}^n\frac{\sinh(2rx_i)}{x_i}\geq r^n e^{-2\langle x,v\rangle}.
\end{equation*}
Thus,
\begin{equation*}
    \int_{\R^n}|g(x)|^2e^{-2\langle x,v\rangle}\dif x\leq r^{-n}|K|\int_{\R^n}|g(x)|^2 e^{\tilde{h}_K(-2x)}\dif x= r^{-n} \|g\|_{L^2(\tilde{h}_K)}^2,
\end{equation*}
proving the lemma.
\end{proof}

\subsubsection{\texorpdfstring{$\mathrm{PW}$}{TEXT} maps to \texorpdfstring{$A^2(T_K)$}{TEXT}}

By Propositions \ref{pw-prop1}--\ref{pw-prop2}, $\mathrm{PW}(L^2(\R^n, \tilde{h}_K))\subset L^2(T_K)$. To show that the image of $\mathrm{PW}$ is in fact in $A^2(T_K)$ it remains to show that $\mathrm{PW}(g)$ is holomorphic. A similar theorem, in a more general setting, was shown by Hultgren \cite[Theorem 3]{hultgren}. It is useful
to first show $\mathrm{PW}(g)$ is continuous.

\begin{lemma}\label{cont-lemma}
For $g\in L^2(\R^n, \tilde{h}_K)$, $\mathrm{PW}(g)$ is continuous on $T_K$.
\end{lemma}
\begin{proof}
Fix $w= \xi +\sqrt{-1}v\in T_K$ and $\delta>0$ such that $w+z\in T_K$ for all $z= u+\sqrt{-1}y\in B_2^{2n}(0,2\delta)$. By Cauchy--Schwarz, 
\begin{equation}\label{cont_eq2}
    \begin{aligned}
        |\mathrm{PW}(g)(w+z) &-\mathrm{PW}(g)(w)|\\ &= \left| \frac{1}{(2\pi)^\frac{n}{2}} 
        \int_{\R^n}g(x) (e^{\sqrt{-1}\langle x,w+z\rangle}- e^{\sqrt{-1}\langle x,w\rangle})\dif x\right|\\
        &\leq 
        \left( \frac{\int_{\R^n}|g(x)|^2 e^{\tilde{h}_K(-2x)}\dif x}{(2\pi)^n}\right)^\frac12\!\! \left(\int_{\R^n} |e^{\sqrt{-1}\langle x,w+z\rangle}- e^{\sqrt{-1}\langle x,w\rangle}|^2 e^{-\tilde{h}_K(-2x)}\dif x \right)^\frac12. 
    \end{aligned}
\end{equation}
Moreover, let $r>0$ such that $v+y+[-r,r]^n\subset \mathrm{int}\,K$ for all $z= u+iy\in B_2^{2n}(0,\delta)$. By (\ref{nazeq1}),
\begin{equation*}
    \begin{aligned}
        \int_{\R^n} |e^{\sqrt{-1}\langle x,w+z\rangle} &- e^{\sqrt{-1}\langle x,w\rangle}|^2 e^{-\tilde{h}_K(-2x)}\dif x 
        \\
        &\leq \int_{\R^n} \left(2|e^{\sqrt{-1}\langle x,w+z\rangle}|^2+ 2|e^{\sqrt{-1} \langle x,w\rangle}|^2\right) e^{-\tilde{h}_K(-2x)}\dif x\\
        &= 2\int_{\R^n} e^{-2\langle x,v+y\rangle} e^{-\tilde{h}_K(-2x)}\dif x
        + 2\int_{\R^n} e^{-2\langle x,v\rangle}e^{-\tilde{h}_K(-2x)}\dif x\\
        &\leq 4|K|\left( \frac{\pi^2}{8r^2}\right)^n,
    \end{aligned}
\end{equation*}
that is finite and independent of $z$. So, dominated convergence applies 
\cite[\S 2.24]{folland},
\begin{equation}\label{cont_eq1}
\begin{aligned}
\lim_{z\to 0}\int_{\R^n} |e^{\sqrt{-1}\langle x,w+z\rangle}- e^{\sqrt{-1}\langle x,w\rangle}|^2 e^{-\tilde{h}_K(-2x)}\dif x &= \\
 \int_{\R^n} \lim_{z\to 0} |e^{\sqrt{-1}\langle x,w+z\rangle}- e^{\sqrt{-1}\langle x,w\rangle}|^2 e^{-\tilde{h}_K(-2x)}\dif x&=0.
\end{aligned}
\end{equation}
From (\ref{cont_eq1}) and (\ref{cont_eq2}) it follows $\lim_{z\to 0}|\mathrm{PW}(g)(w+z)-\mathrm{PW}(g)(w)|=0$, thus $\mathrm{PW}(g)$ is continuous.
\end{proof}

\begin{proof}[Proof of Proposition \ref{pw-prop3}]
Let $g\in L^2(\R^n, \tilde{h}_K)$. 
By Propositions \ref{pw-prop1}, \ref{pw-prop2}, $\mathrm{PW}(g)\in L^2(T_K)$. To show $\mathrm{PW}(g)$ is holomoprhic it suffices to show it is holomorphic in each variable separately. As a result, let us take $n=1$. By Lemma \ref{cont-lemma}, $\mathrm{PW}(g)$ is continuous and hence by Morera's
theorem \cite[p. 122]{ahlfors} it suffices to show that for any closed smooth curve
$\gamma: [0,1]\to T_K$
the integral
$    \int_\gamma \mathrm{PW}(g)(w)\dif w$
vanishes. Let $\gamma(s)= (x(s), y(s))$. Since, the image of $\gamma$ is compact, there exists $r>0$ small enough so that $\gamma(s)+ [-r,r]^2\subset T_K$, for all $s\in [0,1]$
(here we used that \eqref{TKEq}
involves the {\it interior} of $K$). Thus, (\ref{nazeq1}) holds for all $y(s)$ and $s\in [0,1]$. By Cauchy--Schwarz and (\ref{nazeq1}),
\begin{equation*}
\begin{aligned}
    \left|\int_\gamma\int_{\R} |g(x)e^{\sqrt{-1} xw}|\dif x\dif w\right| &\leq \int_0^1\int_{\R}|g(x)| e^{-xy(s)} |\gamma'(s)|\dif x\dif s\\
    &\leq \int_0^1\! \left( \int_{\R}|g(x)|^2 e^{\tilde{h}_K(-2x)}\dif x\!\right)^\frac12\!\! \left(\int_{\R} e^{-2x y(s)} e^{-\tilde{h}_K(-2x)}\dif x\!\right)^\frac12\! |\gamma'(s)|\dif s\\
    &\leq\mathrm{length}(\gamma) \|g\|_{L^2(T_K)} \frac{\pi}{r\sqrt{8}}.
\end{aligned}
\end{equation*}
Since $g\in L^2(\R^n,\tilde{h}_K)$, 
it follows that $\int_\gamma \int_{\R}|g(x)e^{\sqrt{-1} xw}|\dif x\dif w$ is finite.
Thus, the order of integration in $\int \mathrm{PW}(g)(w)\dif w$
can be changed, i.e., by (\ref{f2}),
\begin{equation*}
    \int_\gamma \mathrm{PW}(g)(w)\dif w= (2\pi)^{-\frac{n}{2}}\int_{\gamma}\int_{\R}g(x)e^{\sqrt{-1} xw}\dif x\dif w= (2\pi)^{-\frac{n}{2}}\int_{\R}g(x)\int_\gamma e^{\sqrt{-1}xw}\dif w\dif x=0, 
\end{equation*}
because for each $x$, $w\mapsto e^{\sqrt{-1}xw}$ is holomorphic and $\gamma$ is closed \cite[p. 122]{ahlfors}.
\end{proof}

\subsubsection{\texorpdfstring{$\mathrm{PW}$}{TEXT} is surjective}\label{nazsubsub12}

By Proposition \ref{pw-prop2}, $\mathrm{PW}$ is an isometry to its image in $A^2(T_K)$. To show it surjects onto $A^2(T_K)$, as $L^2(\R^n, \tilde{h}_K)$ is complete, it suffices to show that a dense subset of $A^2(T_K)$ is contained in $\mathrm{PW}(L^2(\R^n, \tilde{h}_K))$
(isometry implies that $\mathrm{PW}$ maps Cauchy sequences in
$L^2(\R^n, \tilde{h}_K)$ to Cauchy
sequences in $A^2(T_K)$).  
The key technical result is Lemma \ref{dense_in_A^2} saying that any $f\in A^2(T_K)$ can be approximated by $\{F_j\}_j\in A^2(T_K)$ such that the $\xi$-Fourier transform of $F_j(\xi+\sqrt{-1}y)$ is compactly supported for all $y\in\mathrm{int}\,K$. The following technical lemma
(augmenting the brief discussion by Berndtsson \cite[p. 405]{berndtsson2})
is required
to carry out such an approximation.

\begin{lemma}\label{bob_intermediate_lemma}
For $f\in A^2(T_K)$ and $\eta\in L^1(\R^n)$, 
\begin{equation*}
    F(w)\defeq \int_{\R^n}f(w-u) \eta(u)\dif u, 
\end{equation*}
is holomorphic in $T_K$. 
\end{lemma}
\begin{proof}
{\it Step 1: the integral is bounded.}
To show that $F(w)\in \C$ for all $w\in T_K$, set $w_0:= \xi_0+\sqrt{-1}y_0\in T_K$, and pick $\e>0$ such that $B_2^{2n}(w_0, \e)\subset T_K$. Since $T_K$ is a tube domain, $B_2^{2n}(\xi+\sqrt{-1}y_0, \e)\subset T_K$, for all $\xi\in \R^n$. As $f$ is holomorphic, $|f|^2$ is subharmonic,
\begin{equation*}\label{f_local_bound}
    |f(\xi+\sqrt{-1}y_0)|^2\leq \frac{1}{\e^{2n}|B_2^{2n}(0,1)|}\int_{B_2^{2n}(\xi+\sqrt{-1}y_0, \e)}|f(w)|^2\dif\lambda(w)\leq \frac{\|f\|_{L^2(T_K)}^2}{\e^{2n}|B_2^{2n}(0,1)|}, 
\end{equation*}
for all $\xi\in \R^n$. Thus,
\begin{equation}\label{F_local_bound}
    \left| \int_{\R^n}f(w_0-u) \eta(u)\dif u\right|\leq \int_{\R^n} |f(w_0-u)| |\eta(u)|\dif u\leq \frac{\|f\|_{L^2(T_K)} \|\eta\|_1}{\e^{n} \sqrt{|B_2^{2n}(0,1)|}}, 
\end{equation}
which shows $F$ is $\C$-valued. 

\noindent {\it Step 2: verifying Morera's criterion.}
Holomorphicity follows as in the proof of Proposition \ref{pw-prop3}. Let $\gamma: [0,1]\to T_K$ be a closed curve in $T_K$. Since its image is compact there exists $\e>0$ such that $B_2^{2n}(\gamma(t), \e)\subset T_K$ for all $t\in [0,1]$. It follows from (\ref{F_local_bound}), that $F$ is bounded on $\gamma([0,1])$, and hence by the holomorphicity of $f$ and (\ref{f2}), 
\begin{equation}\label{mor_c}
    \int_\gamma F \dif w= \int_{\R^n}\left(\int_\gamma f(w-u)\dif w\right) h(u)\dif u=0.
\end{equation}
{\it Step 3: continuity.}
It remains to show that $F$ is continuous since then by (\ref{mor_c}) and Morera's theorem $F$ is holomorphic \cite[p. 122]{ahlfors}. For $w\in T_K$, let $\e>0$ such that $B_2^{2n}(w,2\e)\subset T_K$ and $z\in B_2^{2n}(0,\e)$. As in (\ref{F_local_bound}),
\begin{equation*}
    |F(w+z)- F(w)|= \left|\int_{\R^n}[f(w+z-u)- f(w-u)] \eta(u) \dif u\right|\leq \frac{2\|f\|_{L^2(T_K)}\|\eta\|_{1}}{\e^{n} \sqrt{|B_2^{2n}(0,1)|}}, 
\end{equation*}
because $B_2^{2n}(w+z, \e)\subset T_K$ for all $z\in B_2^{2n}(0,\e)$. As a result, dominated convergence applies \cite[\S 2.24]{folland},
\begin{equation*}
    \lim_{z\to 0}\left[ F(w+z)-F(w)\right]= 
    \int_{\R^n}\lim_{z\to 0}
    [f(w+z-u)-f(w-u)]\eta(u)\dif u= 0,
\end{equation*}
because $f$ is holomorphic, and hence continuous. 

Since $F$ is continuous and (\ref{mor_c}) holds, by Morera's theorem $F$ is holomorphic \cite[p. 122]{ahlfors}.
\end{proof}

For $f\in A^2(T_K)$ and $y\in\mathrm{int}\,K$, denote
$$
f_y(\xi)\defeq f(\xi+\sqrt{-1}y).
$$
Berndtsson claims in a more general setting (replacing $A^2(T_K)$ by $A^2(e^{-2\phi})$, for $\phi$ a convex function) that the class of functions in $f\in A^2(T_K)$ with compactly supported Fourier transform $\hat{f}_y$ for at least one $y\in \mathrm{int}\,K$ is dense in $A^2(T_K)$ \cite[p. 405]{berndtsson2} and gives a brief sketch of a proof. Amplifying his ideas, set
\begin{equation}
\label{CcalEq}
    \mathcal{C}\defeq \{f\in A^2(T_K): \widehat{f_y} \text{ is compactly supported for all } y\in\mathrm{int}\,K\}. 
\end{equation}
\begin{lemma}\label{dense_in_A^2}
$\mathcal{C}$
is dense in $A^2(T_K)$. 
\end{lemma}
\begin{proof}
Set
\begin{equation*}
    \chi (x)\defeq \begin{cases}
    e^{-\frac{1}{1-|x|^2}}, |x|<1, \\
    0, \text{ otherwise}. 
    \end{cases}
\end{equation*}
Note $\chi\in C^\infty(\R^n)$, is supported on $\overline{B_2^n(0,1)}$ with $0\leq \chi\leq 1$, so $\chi \in L^1(\R^n)$. Let also
\begin{equation*}
    \psi(x)\defeq (\chi\ast\chi)(x)= \int_{\R^n}\chi(x-u)\chi(u)\dif u,
\end{equation*}
is smooth, non-negative, supported on $\overline{B_2^n(0,2)}$, with $\hat{\psi}= (\hat{\chi})^2\geq 0$ \cite[Theorem 8.22(c)]{folland}. Moreover, since $0\leq \chi\leq 1$, $\psi$ is bounded by
\begin{equation}\label{nazeq2}
    \psi(x)= \int_{\R^n}\chi(x-u)\chi(u)\dif u\leq \int_{\R^n}\chi(u)\dif u= \|\chi\|_{L^1}.
\end{equation}
By (\ref{planch_eq}), $\hat{\chi}\in L^2(\R^n)$ since $\chi\in L^1(\R^n)\cap L^2(\R^n)$, with $\|\hat{\chi}\|_{L^2}= (2\pi)^{n/2} \|\chi\|_{L^2}$. As a result, $\hat{\psi}\in L^1(\R^n)$ with $\|\hat\psi\|_{L^1}= \|\hat{\chi}\|_{L^2}^2= (2\pi)^n\|\chi\|_{L^2}^2$. For $\e>0$, let
\begin{equation*}
    \eta_\e(x)\defeq \frac{\e^n}{\int\hat{\psi}} \hat{\psi}\left(x/\e\right)= \frac{\e^n}{(2\pi)^n \|\chi\|_{L^2}^2} \hat{\chi}\left( x/\e\right)^2\in C^\infty(\R^n). 
\end{equation*}
Note $\eta_\e$ is non-negative and
\begin{equation}\label{nazeq101}
    \int\eta_\e=1. 
\end{equation}
By (\ref{fourier_inv}), 
\begin{equation*}
    \widehat{\eta_\e}(\xi)\defeq \int_{\R^n}\eta_\e(x)e^{-\sqrt{-1}\langle \xi,x\rangle}\dif x= \frac{\e^n}{\int \hat{\psi}}(2\pi)^n \psi\left(-\xi/\e\right),
\end{equation*}
is smooth and supported on $\overline{B_2^n(0,2\e)}$. Let $f\in A^2(T_K)$. By Lemma \ref{bob_intermediate_lemma},
\begin{equation*}
    f_\e(w)\defeq \int_{\R^n} f(w-u) \eta_\e(u)\dif u,
\end{equation*}
is holomorphic in $T_K$.
Moreover, by Cauchy--Schwarz, (\ref{tonelli}), and (\ref{nazeq101}),
\begin{equation}\label{naz_eqd1}
\begin{aligned}
    \|f_\e\|_{L^2(T_K)}^2=\int_{T_K}|f_\e(w)|^2\dif \lambda(w)& = \int_{T_K}\left| \int_{\R^n} f(w-u)\eta_\e (u)\dif u\right|^2 \dif\lambda(w)\\
    &\leq \int_{T_K}\left(\int_{\R^n} |f(w-u)|^2 \eta_\e(u)\dif u\right) \left( \int_{\R^n} \eta_\e(u)\right)\dif\lambda(w)
    \\&= \int_{T_K}\int_{\R^n}|f(w-u)|^2 \eta_\e(u)\dif u\dif \lambda(w)
    \\&= \|f\|^2_{L^2(T_K)} \int \eta_\e= \|f\|_{L^2(T_K)}^2.
\end{aligned}
\end{equation}
Therefore, $f_\e\in A^2(T_K)$. Furthermore, $f_\e$ has compactly supported Fourier transform for all $y\in \mathrm{int}\,K$, since $(f_\e)_y= f_y\ast \eta_\e$, thus $\widehat{(f_\e)_y}= \widehat{f_y}\widehat{\eta_\e}$ \cite[Theorem 8.22]{folland}, which is compactly supported because $\widehat{\eta_\e}$ is. 

It remains to show that $f_\e$ $L^2$-converges to $f$. Observe,
\begin{equation*}
\begin{aligned}
    \|f_\e-f\|^2_{L^2(T_K)}&= \int_{\mathrm{int}\,K}\int_{\R^n} |f_\e(\xi+\sqrt{-1}y)- f(\xi+\sqrt{-1}y)|^2\dif \xi\dif y\\
    &= \int_{\mathrm{int}\,K} \|(f_\e)_y- f_y\|^2_{L^2(\R^n)}\dif y.
\end{aligned}
\end{equation*}
But,
$\lim_{\e\to 0}\|(f_\e)_y-f_y\|_{L^2(\R^n)}^2= 0$, for almost all $y\in\mathrm{int}\,K$ \cite[Theorem 8.14]{folland}, and
by (\ref{naz_eqd1}) 
$\|(f_\e)_y-f_y\|_{L^2(\R^n)}^2\leq 4\|f_y\|_{L^2(\R^n)}^2$,
 that is integrable because
$
\int_{\mathrm{int}\,K}\|f_y\|_{L^2(\R^n)}^2\dif y= \|f\|_{L^2(T_K)}^2.
$
Combining this and dominated convergence \cite[\S 2.24]{folland},
gives
$
    \lim_{\e\to 0}\|f_\e-f\|_{L^2(T_K)}=~0.
$
\end{proof}

The next argument is due to Berndtsson \cite[pp. 404--405]{berndtsson2}. 
\begin{proof}[Proof of Proposition \ref{pw-prop4}]
By Propositions \ref{pw-prop1}--\ref{pw-prop3}, $\mathrm{PW}: L^2(\R^n, \tilde{h}_K)\to A^2(T_K)$ is an isometry. Thus, as remarked at the beginning of \S\ref{nazsubsub12}, to show it is surjective
it suffices to show that its image is dense. By Lemma \ref{dense_in_A^2}, it is enough to prove the theorem for $f\in A^2(T_K)$ with $\widehat{f_y}$ compactly supported for all $y\in\mathrm{int}\,K$. Let $f\in A^2(T_K)$ be such a function, and write $f_y(\xi)\defeq f(\xi+\sqrt{-1}y)$. Since 
\begin{equation*}
    \|f\|_{L^2(T_K)}^2= \int_{\mathrm{int}\,K}\int_{\R^n}|f(\xi+\sqrt{-1}y)|^2\dif x\dif y<\infty, 
\end{equation*}
$f_y$ is $L^2(\R^n)$-integrable for almost all $y\in\mathrm{int}\,K$. In particular, there exists some $y_0\in \mathrm{int}\,K$ such that $f_{y_0}(\xi)\defeq f(\xi+\sqrt{-1} y_0)$ is $L^2(\R^n)$-integrable. By Fourier inversion (\ref{fourier_inv}),
\begin{equation}\label{pw_eq1}
    f_{y_0}(\xi)= \frac{1}{(2\pi)^n}\int_{\R^n}\widehat{f_{y_0}}(x)e^{\sqrt{-1}\langle \xi,x\rangle}\dif x.
\end{equation}
By assumption, $\widehat{f_{y_0}}\in L^2(\R^n)$ is compactly supported. Therefore, $\widehat{f_{y_0}}(x)e^{\langle x,y_0\rangle}\in L^2(\R^n)$ is also compactly supported and, in particular, lies in $L^2(\R^n, \tilde{h}_K)$. By Propositions \ref{pw-prop1}--\ref{pw-prop3} then,
\begin{equation}\label{pw_eq2}
    F(w)\defeq \mathrm{PW}((2\pi)^{-\frac{n}{2}} \widehat{f_{y_0}}(x)e^{\sqrt{-1}\langle x,y_0\rangle})(w)= \frac{1}{(2\pi)^n}\int_{\R^n} \widehat{f_{y_0}}(x)e^{\langle x,y_0\rangle} e^{\sqrt{-1}\langle x,w\rangle}\dif x,
\end{equation}
is well-defined and holomorphic in $T_K$. 
By (\ref{pw_eq1}) and (\ref{pw_eq2}),
\begin{equation*}
    F(\xi+\sqrt{-1}y_0)=  \frac{1}{(2\pi)^n}\int_{\R^n}\widehat{f_{y_0}}(x) e^{\sqrt{-1}\langle \xi,x\rangle}\dif x= f(\xi+\sqrt{-1}y_0), 
\end{equation*}
i.e., $F$ agrees with $f$ on $\R^n\times \{y_0\}$. Since they are both holomorphic, by analytic continuation $f\equiv F= \mathrm{PW}((2\pi)^{-\frac{n}{2}} \widehat{f_{y_0}}(x)e^{\sqrt{-1}\langle x,y_0\rangle})$, as desired. 
\end{proof}

\begin{remark}
In fact, restricting to the larger family (recall \eqref{CcalEq})
\begin{equation*}
    \widetilde{\mathcal{C}}\defeq \{f\in A^2(T_K): \widehat{f_y} \text{ is compactly supported for some } y\in\mathrm{int}\,K\}
    \supset \mathcal{C},
\end{equation*}
suffices for the proof of Proposition \ref{pw-prop4} above. This is because for $y_0\in\mathrm{int}\,K$ such that $\widehat{f_{y_0}}$ is compactly supported, by (\ref{planch_eq}), $\xi\mapsto f(\xi+iy_0) \in L^2(\R^n)$,
and hence (\ref{pw_eq1}) holds.
\end{remark}

\begin{remark}
Perhaps a more intuitive proof for Proposition \ref{pw-prop4} would be the following. Take an $f\in A^2(T_K)$. Assume that $f_y\in L^2(\R^n)$
for all $y\in \mathrm{int}\,K$. By (\ref{fourier_inv}), 
\begin{equation}\label{naz_eq102}
    f(\xi+\sqrt{-1}y)= \frac{1}{(2\pi)^n}\int_{\R^n} \widehat{f_{y}}(x) e^{\sqrt{-1}\langle x,\xi\rangle}\dif x= \frac{1}{(2\pi)^n}\int_{\R^n} \widehat{f_y}(x) e^{\langle x,y\rangle} e^{\sqrt{-1}\langle x, \xi+\sqrt{-1}y\rangle}\dif x.
\end{equation}
In view of (\ref{naz_eq102}), let $g(x,y)\defeq \widehat{f_y}(x) e^{\langle x,y\rangle}$. For $g$ independent of $y$, $f= \mathrm{PW}(g)$ as desired. This is where the holomorphicity of $f$ comes into play. By (\ref{fourier_transform}) and (\ref{naz_eq102}),
\begin{equation*}
    g(x,y)= \frac{1}{(2\pi)^n}\int_{\R^n} f(\xi+ \sqrt{-1}y) e^{\langle x,y\rangle} e^{-\sqrt{-1}\langle \xi, x\rangle}\dif \xi. 
\end{equation*}
Assuming that one can differentiate under the integral sign, 
\begin{equation}\label{naz_eq103}
\begin{aligned}
    \frac{\partial g}{\partial y}&= \frac{1}{(2\pi)^n} \int_{\R^n}\left(\frac{\partial f}{\partial y}+ x f\right) e^{\langle x,y\rangle} e^{-\sqrt{-1}\langle \xi, x\rangle}\dif \xi. 
\end{aligned}
\end{equation}
Moreover, $f$ is holomorphic, thus $\partial f/\partial \xi= -\sqrt{-1}\partial f/\partial y$, i.e., 
\begin{equation}\label{naz_eq104}
    \begin{aligned}
        \int_{\R^n}\frac{\partial f}{\partial y} e^{\langle x,y\rangle} e^{-\sqrt{-1}\langle \xi, x\rangle}\dif \xi&= \sqrt{-1}\int_{\R^n}\frac{\partial f}{\partial \xi}e^{\langle x,y\rangle} e^{-\sqrt{-1}\langle \xi, x\rangle}\dif \xi  \\
        &= -\int_{\R^n} xf e^{\langle x,y\rangle} e^{-\sqrt{-1}\langle \xi, x\rangle}\dif \xi, 
    \end{aligned}
\end{equation}
by integration by parts. It follows from (\ref{naz_eq103}) and (\ref{naz_eq104}) that $\partial g/\partial y=0$, that is, $g$ is independent of $y$. Nonetheless, several assumptions were made, that may not hold in general, i.e., $f_y$ is $L^2$-integrable for all $y\in \mathrm{int}\,K$ and taking the derivatives under the integral,
but this can be made rigorous
\cite[p. 94]{blocki}. 
\end{remark}

\subsection{Proof of Proposition \ref{nazprop1}}
\label{nazsub2}

With the Paley--Wiener correspondence established (Theorem \ref{pw-thm}) the proof of the 
lower bound on the Mahler volume in terms of the Bergman kernel
conceptually proceeds as follows. 

\begin{itemize}
    \item For $a\in\mathrm{int}\,K$, $\M(K-a)$ is the product of $|K|$ with $n!|(K-a)^\circ|$. The latter equals to the integral of $e^{-h_{K-a}(x)}$, where $h_K$ is the support function of $K$ (Claim \ref{integral_polar_support}). 
    
    \item Jensen's inequality provides a lower bound $e^{h_{K-a}(x)}\leq 2^n e^{\tilde{h}_{K-a}(2x)}$ 
    (Lemma \ref{J_K-lower_bound}).
    
     \item Using the Paley--Wiener correspondence established in the previous section one may verify a formula for $\mathcal{K}_{T_K}(z,w)$ so that $(2\pi)^n |K|\mathcal{K}_{T_K}(\sqrt{-1}a, \sqrt{-1}a)= \int_{\R^n} e^{-\tilde{h}_{K-a}(-2x)}\dif x$ on the diagonal
     (Lemma \ref{bk-tube}). 
     
     \item By the first step, $\M(K-a)= |K|\int_{\R^n} e^{-h_{K-a}}$ which, by the previous two steps, is bounded below by $\pi^n|K|^2 \mathcal{K}_{T_K}(\sqrt{-1}a, \sqrt{-1}a)$ proving Proposition \ref{nazprop1}.
\end{itemize}

The following is a well-known formula for $|K^\circ|$ in terms of $h_K$ \cite[(2.3)]{hultgren}. 
\begin{claim}\label{integral_polar_support}
For a convex body $K\subset \R^n$ 
satisfying (\ref{intKeq}),
$
    \int_{\R^n}e^{-h_K(y)}\dif y= n!|K^\circ|. 
$
\end{claim}

Jensen's inequality gives a lower bound on $e^{-h_{K-b(K)}(x)}$ in terms of $\tilde{h}_K$. For a subset $S\subset \R^n$ denote by 
\begin{equation*}
    \mathbf{1}_S(x)\defeq \begin{cases}
    1 & \text{ for } x\in S, \\
    0 & \text{ otherwise}.
    \end{cases}
\end{equation*}
\begin{lemma}\label{J_K-lower_bound}
For $K\subset \R^n$ a convex body and $a\in K$,
$
    e^{h_{K-a}(x)}\leq 2^n e^{\tilde{h}_{K-a}(2x)}. 
$
\end{lemma}
\begin{proof}
Assume $b(K)=0$ for a moment. Note that
\begin{equation*}
    a= b(K)+a =\frac{1}{|K|}\int_K (u+a)\dif u= \frac{1}{|K|}\int_{K+a} v\dif v.
\end{equation*}
Fix $y\in K$, $x\in \R^n$, and let $F(u)\defeq e^{\langle u,x\rangle}$. By Jensen's inequality \cite[Remark A.2.3]{artstein-giannopoulos-milman}, for the probability measure $\frac{1}{|K|}\int_{K+a+y}\dif v$ and the convex function $F$, 
\begin{equation*}
    \begin{aligned}
        e^{\langle y-a, x\rangle}&= e^{-2\langle a,x\rangle} e^{\langle y+a,x \rangle}= e^{-2\langle a,x\rangle} e^{\langle \frac{1}{|K|}\int_K (u+y+a)\dif u, x\rangle}= e^{-2\langle a,x\rangle} F\left( \frac{1}{|K|}\int_K u+y+a \dif u\right) \\
        &\leq e^{-2\langle a,x\rangle} \frac{1}{|K|}\int_K F(u+y+a)\dif u=  e^{-2\langle a,x\rangle}\frac{1}{|K|}\int_K e^{\langle u+y+a, x\rangle}\dif u \\
        &= e^{-2\langle a,x\rangle}\frac{1}{|K|}\int_K e^{2\langle \frac{u}{2}+\frac{y+a}{2}, x\rangle}\dif u= e^{-2\langle a,x\rangle}\frac{2^n}{|K|}\int_{\frac{K}{2}+ \frac{y+a}{2}} e^{2\langle v,x\rangle}\dif v\\
        &\leq e^{-2\langle a,x\rangle} \frac{2^n}{|K|}\int_K e^{2\langle v,x\rangle}\dif v= \frac{2^n}{|K|}\int_K e^{\langle v-a, 2x\rangle}\dif v= 2^n e^{\tilde{h}_{K-a}(2x)}, 
    \end{aligned}
\end{equation*}
because $y,a\in K$ thus $(y+a)/2\in K$ and hence $K/2+ (y+a)/2\subset K$. Taking supremum over all $y\in K$ yields $e^{h_{K-a}(x)}\leq 2^n e^{\tilde{h}_{K-a}(2x)}$ as desired. 

In general, for any $a\in K$ write $a= a-b(K)+ b(K)$. By the previous case, since $a-b(K)\in K-b(K)$ and $b(K-b(K))=0$
\begin{equation*}
    e^{h_{K-a}(x)}= e^{h_{K-b(K)- (a-b(K))}(x)}\leq 2^n e^{\tilde{h}_{K-b(K)-(a-b(K))}(2x)}= e^{\tilde{h}_{K-a}(2x)}, 
\end{equation*}
as desired.
\end{proof}

The left-hand side of the inequality in Lemma \ref{J_K-lower_bound} appears in the integral representation of $\mathcal{K}_{T_K}(\sqrt{-1}a, \sqrt{-1}a)$ since on the diagonal one may explicitly compute 
\begin{equation}\label{nazeq121}
\begin{aligned}
    \mathcal{K}_{T_K}(\sqrt{-1} a, \sqrt{-1}a)&= \frac{1}{(2\pi)^n|K|}\int_{\R^n} e^{-\tilde{h}_{K-a}(-2x)}\dif x.
\end{aligned}
\end{equation}
This follows from the general formula for the Bergman kernel of a tube domain of a convex body $\mathcal{K}_{T_K}(z,w)$ (Lemma \ref{bk-tube} below) \cite[(1.2)]{hsin} \cite[Theorem 2.6]{rothaus},
and the following computation: 
\begin{equation*}
\begin{aligned}
  e^{2\langle a, x\rangle+ \tilde{h}_K(-2x)} &= e^{2\langle a,x\rangle} |K|\int_{K} e^{-2\langle x,y\rangle}\dif y= |K|\int_{K}e^{-2\langle x, y-a\rangle}\dif y\\
  &= |K-a|\int_{K-a} e^{-2\langle x,y\rangle}\dif y= e^{\tilde{h}_{K-a}(-2x)}.
 \end{aligned}
\end{equation*}

\begin{lemma}\label{bk-tube}
For a convex body $K\subset\R^n$, 
    \begin{equation*}
        \mathcal{K}_{T_K}(z,w)= \frac{1}{(2\pi)^n |K|}\int_{\R^n} e^{\sqrt{-1}\langle z-\overline{w}, x\rangle- \tilde{h}_K(-2x)}\dif x. 
    \end{equation*}
\end{lemma}
\begin{proof}
For $z= \xi+\sqrt{-1}y\in T_K$ and $w= a+\sqrt{-1} b\in T_K$, since $K$ is convex $(y+b)/2\in \mathrm{int}\,K$. Take $r>0$ such that $(y+b)/2+[-r,r]^n\subset \mathrm{int}\,K$. By (\ref{nazeq1}), 
\begin{equation}\label{nazeq122}
    \begin{aligned}
        \int_{\R^n}\left| e^{\sqrt{-1}\langle z-\overline{w}, x\rangle- \tilde{h}_K(-2x)}\right|\dif x&= \int_{\R^n} e^{-\langle y+b, x\rangle- \tilde{h}_{K}(-2x)}\dif x\\
        &= \int_{\R^n} e^{-2\langle\frac{y+b}{2}, x\rangle- \tilde{h}_{K}(-2x)}\dif x\leq \left( \frac{\pi^2}{8r^2}\right)^n.
    \end{aligned}
\end{equation}
As a result, by (\ref{nazeq122}),
\begin{equation}\label{nazeq123}
\begin{aligned}
    F(z,w)&\defeq \frac{1}{(2\pi)^n |K|}\int_{\R^n}e^{\sqrt{-1}\langle z-\overline{w}, x\rangle- \tilde{h}_{K}(-2x)}\dif x\\
    &= \frac{1}{(2\pi)^n |K|} \int_{\R^n} e^{-\langle y+b,x\rangle- \sqrt{-1}\langle a,x\rangle- \tilde{h}_{K}(-2x)} e^{\sqrt{-1}\langle \xi, x\rangle}\dif x
\end{aligned}
\end{equation}
converges in $\C$ for all $z,w\in T_K$. In particular, by Remark \ref{l2condition}, 
\begin{equation*}
   G(x)\defeq \frac{1}{|K|} e^{-\langle y+b, x\rangle- \sqrt{-1}\langle a,x\rangle- \tilde{h}_K(-2x)}, 
\end{equation*}
is $L^2$-integrable and, by (\ref{nazeq123}), it is the Fourier transform of $\xi\mapsto F(\xi+iy, w)$. Therefore, by (\ref{planch_eq}), $\xi\mapsto F(\xi+iy, w)$ is $L^2$-integrable with 
\begin{equation}\label{nazeq126}
    \int_{\R^n}|F(\xi+iy, w)|^2\dif\xi= (2\pi)^{-n} \int_{\R^n}|G(x)|^2\dif x= \frac{1}{(2\pi)^n |K|^2} \int_{\R^n} e^{-2\langle y+b, x\rangle-2\tilde{h}_{K}(-2x)}\dif x. 
\end{equation}
Integrating (\ref{nazeq126}) with respect to $y$, by (\ref{tonelli}) (since the integrand is positive), 
\begin{equation*}
\begin{aligned}
    \int_{T_K}|F(z,w)|^2\dif\lambda(z)&= \frac{1}{(2\pi)^n|K|^2} \int_{\mathrm{int}\,K} \int_{\R^n} e^{-2\langle y+b, x\rangle- 2\tilde{h}_K(-2x)}\dif x \dif y\\
    &= \frac{1}{(2\pi)^n|K|^2} \int_{\R^n} e^{-2\langle x,b\rangle-2\tilde{h}_{K}(-2x)}\left( \int_{\R^n} e^{-2\langle y,x\rangle}\dif y\right)\dif x \\
    &= \frac{1}{(2\pi)^n|K|^2} \int_{\R^n} e^{-2\langle x,b\rangle-2\tilde{h}_{K}(-2x)} |K| e^{\tilde{h}_K(-2x)}\dif x \\
    &= \frac{1}{(2\pi)^n|K|} \int_{\R^n} e^{-2\langle x,b\rangle- \tilde{h}_K(-2x)}\dif x= F(w,w)
\end{aligned}
\end{equation*}
is finite by (\ref{nazeq123}), i.e., $z\mapsto F(z, w)\in L^2(T_K)$.

Moreover, $F$ enjoys a reproducing property. To see why, let $f\in A^2(T_K)$. By Theorem \ref{pw-thm}, let $g\in L^2(\R^n, \tilde{h}_K)$ be such that $f(z)\defeq \mathrm{PW}(g)(z)= (2\pi)^{-n/2} \int_{\R^n} g(x)e^{\sqrt{-1}\langle z,x\rangle}\dif x$. In particular, $x\mapsto (2\pi)^\frac{n}{2} g(x) e^{-\langle x,y\rangle}$ is the Fourier transform of $\xi\mapsto f(\xi+ \sqrt{-1}y)$. By (\ref{parseval}), 
\begin{equation}\label{nazeq125}
\begin{aligned}
    \int_{\R^n} f(z)\overline{F(z,w)}\dif \xi&= (2\pi)^{-n} \int_{\R^n} (2\pi)^{\frac{n}{2}} g(x) e^{-\langle x,y\rangle} \overline{G(x)}\dif x\\ 
    &= \frac{1}{(2\pi)^{\frac{n}{2}}|K|} \int_{\R^n} g(x) e^{-\langle b,x\rangle+ \sqrt{-1}\langle a,x\rangle- \tilde{h}_K(-2x)}e^{-2\langle x,y\rangle} \dif x \\
    &= \frac{1}{(2\pi)^{\frac{n}{2}}|K|} \int_{\R^n} g(x) e^{\sqrt{-1}\langle w,x\rangle- \tilde{h}_K(-2x)} e^{-2\langle x,y\rangle}\dif x.
\end{aligned}
\end{equation}
Integrating (\ref{nazeq125}) with respect to $y$ over $\mathrm{int}\,K$, by (\ref{f2}),
\begin{equation*}
    \begin{aligned}
        \langle f, F(\cdot, w)\rangle_{L^2(T_K)}&= \int_{T_K} f(z)\overline{F(z,w)}\dif \lambda(z)= \int_{\mathrm{int}\, K}\int_{\R^n} f(z)\overline{F(z,w)}\dif \xi\dif y\\
        &= \frac{1}{(2\pi)^{\frac{n}{2}}|K|} \int_{\mathrm{int}\,K} \int_{\R^n} g(x) e^{\sqrt{-1}\langle w,x\rangle- \tilde{h}_K(-2x)} e^{-2\langle x,y\rangle}\dif x\dif y\\
        &= \frac{1}{(2\pi)^{\frac{n}{2}}|K|} \int_{\R^n} g(x) e^{\sqrt{-1}\langle w,x\rangle-\tilde{h}_K(-2x)}\left( \int_{\mathrm{int}\,K}e^{-2\langle x,y\rangle}\dif y\right)\dif x \\
        &= (2\pi)^{-\frac{n}{2}}\int_{\R^n}g(x) e^{\sqrt{-1}\langle w,x\rangle-\tilde{h}_K(-2x)} e^{\tilde{h}_K(-2x)}\dif x \\
        &= (2\pi)^{-\frac{n}{2}} \int_{\R^n} g(x) e^{\sqrt{-1}\langle w,x\rangle}\dif x= f(w).
    \end{aligned}
\end{equation*}
To justify the use of $(\ref{f2})$, by Cauchy--Schwarz 
\begin{equation*}
    \int_{T_K}|f(z)\overline{F(z,w)}|\dif\lambda(z)\leq \left(\int_{T_K} |f(z)|^2\dif\lambda(z) \right)^\frac12\left(\int_{T_K}|F(z,w)|^2\dif\lambda(z)\right)^\frac12
\end{equation*}
is finite.

As a result, by  
the reproducing properties of $\mathcal{K}_{T_K}$
and $F$,
\begin{equation}\label{naz_eq200}
   F(z,w)= \overline{F(w,z)}= \overline{\langle F(\cdot, z), \mathcal{K}_{T_K}(\cdot, w)\rangle}= \langle \mathcal{K}_{T_K}(\cdot, w), F(\cdot, z)\rangle= \mathcal{K}_{T_K}(z,w),
\end{equation}
because $F$ and $\mathcal{K}_{T_K}$ are holomorphic in the first variable.
\end{proof}

Combining (\ref{nazeq121}) with Claim \ref{integral_polar_support} and Lemma \ref{J_K-lower_bound} proves Proposition \ref{nazprop1}.

\begin{proof}[Proof of Proposition \ref{nazprop1}]
By Claim \ref{integral_polar_support}, Lemma \ref{J_K-lower_bound} and (\ref{nazeq121}),
\begin{equation*}
    \begin{aligned}
        \M(K-a)&= |K||(K-a)^\circ| \\
        &=  |K|\int_{\R^n}e^{-h_{K-a}(x)}\dif x \\
        &\geq  \frac{|K|}{2^n}\int_{\R^n} e^{-\tilde{h}_{K-a}(-2x)} \dif x\\ 
        &=  \frac{|K|}{2^n}(2\pi)^n |K| \mathcal{K}_{T_K}(\sqrt{-1}a, \sqrt{-1}a) \\
        &= \pi^n |K|^2 \mathcal{K}_{T_K}(\sqrt{-1}a, \sqrt{-1}a),
    \end{aligned}
\end{equation*}
as desired
\end{proof}

\section{Estimating the Bergman kernel}\label{nazsub3}\label{sec4}

This section proves Proposition \ref{nazprop2}. Conceptually,
here are the key ideas:
\begin{itemize}
    \item Lemma \ref{equivalent_norms} recalls the standard characterization of the Bergman kernel on the diagonal as a supremum involving $L^2$ holomorphic functions. This reduces the proof of 
    Proposition \ref{nazprop2} to finding such a function that equals $1$ at $\sqrt{-1}b(K)$ and has $L^2$ norm bounded above by $2^n|K|$.
    
    \item Lemma \ref{bk_ai}, proved in \S\ref{sec4.4}, establishes the affine invariance of $\mathcal{B}(K)$. This allows to displace $K$ by a convenient affine transformation for the remainder of the proof.
    
    \item John's theorem is used in Lemma \ref{naz_john_lemma} to place $K$ in better position via an affine transformation, while Santal\'o's inequality ensures good control on $K^\circ$ in the new position. 
    
    \item A weight function $\phi$ (\ref{weight_phi}) on $T_K$ is 
    constructed (Lemma \ref{properties_list}) satisfying the conditions of Proposition \ref{hormander}, with $\phi$  bounded from above, bounded from below away from the origin, and $e^{-\phi}$ is not integrable around the origin. 
    
    \item A smooth bump function $g:\C^n\to \C$ is constructed (Lemma \ref{g_estimate}) with controlled weighted $L^2$-norm over $T_K$ with respect to the weight function of the previous step. While not holomorphic, this bump function has all the other properties one wants in order to estimate the Bergman kernel using Lemma  \ref{equivalent_norms}.
    
    \item H\"ormander's theorem is used to solve for $h$ with $\bar\partial h=-\bar\partial g$
    and ``correct" $g$ to a holomorphic function 
    $f:=g+h$ on $T_K$ with the other properties intact, namely, $f(0)=1$ and bounded weighted $L^2$-norm. 
    This requires several auxiliary estimates, mainly: an earlier upper bound on the weight function (Lemma \ref{properties_list}~(ii)) guaranteeing that the non-weighted $L^2$-norm of $h$ is controlled; control on the volume of the support of $g$ (Lemma \ref{properties_list} (v)) 
     guaranteeing that the non-weighted $L^2$-norm of $g$ is controlled;
     prescribed singularity of $\phi$ at the origin guaranteeing that
     $h(0)=0$. Altogether,
    by Lemma \ref{equivalent_norms}, this yields the desired bound on the Bergman kernel up to subexponential terms (Lemma \ref{affine_nazprop2}). 
    
    \item Tensorization for Bergman kernels is used to eliminate the subexponential terms in the previous bound (Proposition \ref{bk_tensorization}), yielding Proposition \ref{nazprop2}. 
\end{itemize}

\subsection{Bergman kernel on the diagonal}
First, recall the standard characterization of the Bergman kernel on the diagonal in terms of the norm of the evaluation functional $A^2(T_K)\ni f\mapsto f(w)\in \C$ \cite[(5.2)]{hultgren}.
\begin{lemma}\label{equivalent_norms}
For $K\subset \R^n$ a convex body and $w\in T_K$,
    \begin{equation}\label{bk_eq1}
        \mathcal{K}_{T_K}(w,w)= \|\mathcal{K}_{T_K}(\cdot, w)\|^2_{L^2(T_K)}= \sup_{\substack{f\in A^2(T_K)\\ f\neq 0}}\frac{|f(w)|^2}{\|f\|^2_{L^2(T_K)}}.
    \end{equation}
\end{lemma}
\begin{proof}
The first equality follows from (\ref{naz_eq200}).
By Cauchy--Schwarz for any $f\in A^2(T_K)$,
\begin{equation}\label{bk_eq2}
    |f(w)|= \left| \int_{T_K}f(z)\overline{\mathcal{K}_{T_K}(z,w)}\dif\lambda(z)\right|\leq \|\mathcal{K}_{T_K}(\cdot, w)\|_{L^2(T_K)}\|f\|_{L^2(T_K)},
\end{equation}
and hence, by (\ref{bk_eq2}) and the first equality of \eqref{bk_eq1},
\begin{equation*}
    \sup_{\substack{f\in A^2(T_K) \\ f\neq 0}}\frac{|f(w)|^2}{\|f\|_{L^2(T_K)}^2}\leq \|\mathcal{K}_{T_K}(\cdot, w)\|_{L^2(T_K)}^2= \frac{\mathcal{K}_{T_K}(w,w)^2}{\|\mathcal{K}_{T_K}(\cdot, w)\|_{L^2(T_K)}^2}\leq  \sup_{\substack{f\in A^2(T_K) \\ f\neq 0}}\frac{|f(w)|^2}{\|f\|_{L^2(T_K)}^2},
\end{equation*}
since $z\mapsto \mathcal{K}_{T_K}(z,w)\in A^2(T_K)$, proving the second equality in (\ref{bk_eq1}).
\end{proof}

\subsection{Affine invariance of \texorpdfstring{$\mathcal{B}(K)$}{TEXT}}
\label{sec4.4}

The discussion of \S\ref{BergSec} for a tube domain $T_K$ carries over
to any domain $\Omega\subset \C^n$, to yield a Bergman kernel $\mathcal{K}_\Omega(z,w)$, that is the reproducing kernel of the evaluation functional $\mathrm{ev}_{\Omega,w}: A^2(\Omega)\to \C$ at $w\in \Omega$.

The following lemma describes how the Bergman kernel behaves under affine transformations.

\begin{lemma}\label{bergman_affine_transform}
Let $\Omega\subset \C^n$ open domain in $\C^n$ and $T(z)= Az+ b$,
$A\in GL(n,\C)$, $b\in \C^n$. For $z,w\in \Omega$, $\mathcal{K}_{\Omega}(z,w)= \mathrm{det}_\R A \cdot \mathcal{K}_{T\Omega}(Tz, Tw)$.
\end{lemma}

\begin{remark}
A $\C$-linear map $A: \C^n\to \C^n$ can also be viewed as an $\R$-linear map $\R^{2n}\ni (x,y)\mapsto (\mathrm{Re}(A(x+\sqrt{-1}y)), \mathrm{Im}(A(x+\sqrt{-1}y)))\in \R^{2n}$. Denote by $\det_\C A$ the determinant of the former and $\det_\R A$ the determinant of the latter.
Then,
$\det_\R A= |\det_\C A|^2$ \cite[Lemma 2]{CegPer}. 
\end{remark}

\begin{proof}[Proof of Lemma \ref{bergman_affine_transform}]
Let $f\in A^2(T\Omega)$. Then,
\begin{equation*}
    f(Tw)= \int_{T\Omega}f(\zeta) \mathcal{K}_{T\Omega}(\zeta, Tw)\dif\lambda(\zeta)= \int_\Omega f(Tz)\mathcal{K}_{T\Omega}(Tz, Tw) \mathrm{det}_\R A\dif\lambda(z),
\end{equation*}
because $|\det_\R A|= \det_\R A$, since $\det_\R A= |\det _\C A|^2\geq 0$.
On the other hand, since $f\in A^2(T\Omega)$ and $T$ is a holomorphic
map then $f\circ T\in A^2(\Omega)$. So,
\begin{equation*}
    f(Tw)= (f\circ T)(w)= \int_{\Omega} f(Tz) \mathcal{K}_{\Omega}(z,w)\dif\lambda(z).
\end{equation*}
The claim follows by comparing the two equations.
\end{proof}

A direct application of Lemma \ref{bergman_affine_transform} gives the affine invariance of $\mathcal{B}(K)$.

\begin{proof}[Proof of Lemma \ref{bk_ai}]
Let $K\subset \R^n$ be a convex body, and $S(y)= Ay+a$, $A\in GL(n,\R)$, $a\in \R^n$ be an affine transformation.
Consider the embedding of $K$ in $T_K$
as $\{0\}\times\sqrt{-1} K$, and the induced transformation,
still denoted by $S$, $S:\i y\mapsto\i Ay+\i a$.
There is a unique extension of $S$ to a $\C$-linear map on $\R^n+\sqrt{-1}\R^n=\C^n$, that we still denote by $S$,
${S}(z):= Ax +\sqrt{-1}(Ay+a)$. Note,
$\det_\C S= (\det A)^2$.
By (\ref{affine_barycenter}),
\begin{equation*}
S(\sqrt{-1}b(K))= \sqrt{-1}Ab(K)+ \sqrt{-1}a
= \sqrt{-1} b(AK)+ \sqrt{-1}a= \sqrt{-1}b(AK+a)= \sqrt{-1}b(S(K)),
\end{equation*}
and hence, by Lemma \ref{bergman_affine_transform},
\begin{equation*}
    \begin{aligned}
        |K|^2\mathcal{K}_{T_K}(\sqrt{-1}b(K), \sqrt{-1}b(K))&= |K|^2(\det A)^2 \mathcal{K}_{S(T_K)}(S(\sqrt{-1}b(K)), S(\sqrt{-1}b(K))) \\
        &= |S(K)|^2 \mathcal{K}_{T_{S(K)}}(\sqrt{-1}b(S(K)), \sqrt{-1}b(S(K))), 
    \end{aligned}
\end{equation*}
because $S(T_K)= A\R^n+ \sqrt{-1}AK+\sqrt{-1}a= \R^n+ \sqrt{-1}S(K)= T_{S(K)}$, and $|S(K)|= |AK+a| =|\det A||K|$.
\end{proof}

\begin{remark}
Motivated by Lemma \ref{bk_ai}, in a subsequent article \cite{MR2}
we introduce a family of affine invariants
involving $L^p$-versions of the support function
that generalize $\mathcal{B}$. 
To give a glimpse, setting
 $   h_{K,p}(x)\defeq \log[{|K|}^{-1} \int_{K} e^{p\langle x,y\rangle}\dif y]/p, 
$
we introduce the $p$-Mahler volume 
$\M_p(K)\defeq 
|K| \left( \int_{\R^n} e^{-h_K^p(y)}\dif y\right).
$
Then, $ \mathcal{B}(K)= (4\pi)^{-n}\M_1(K).$
This leads to an alternative proof of Lemma \ref{bk_ai},
that simultaneously generalizes to all $p>0$
(as well as to a generalization of Proposition \ref{nazprop1}).
In \cite{MR2} we also study extremizers and monotonicity 
properties of $\M_p$.
 \end{remark}

\subsection{Repositioning \texorpdfstring{$K$}{TEXT}}
\label{repositionSec}

For symmetric bodies, by John's theorem there exist $A\in\mathrm{GL}(n,\R)$ and $r>0$ such that $B_2^n(0,r)\subset AK\subset B_2^n(0,r\sqrt{n})$ and hence $B_2^n(0,1/(r\sqrt{n}))\subset (AK)^\circ\subset B_2^n(0, 1/r)$. However, without the assumption of symmetry one cannot be certain that the maximal ellipsoid contained in $K$ contains $b(K)$. As a result, John's theorem alone does not guarantee a good upper bound for elements of $(K-b(K))^\circ$.
Nonetheless, combining John's theorem \cite[Theorem III]{john} with Santal\'o's inequality \cite[(3.12)]{santalo} yields the following explicit inclusions. 

\begin{lemma}\label{naz_john_lemma}
For a convex body $K\subset\R^n$ with $b(K)=0$, there exists $A\in GL(n, \R)$, $r>0$ and $a\in AK\subset\R^n$ 
such that, 
\begin{align}
\label{johneq1}
    B_2^n(a,r)\subset AK &\subset B_2^n(0,2nr), 
\\    \label{johneq2}
  (AK)^\circ&\subset B_2^n(0,{2n}/{r}).
\end{align}
\end{lemma}
\begin{proof}
For the proof of (\ref{johneq1}) we merely need $0\in\mathrm{int}\,K$. By John's theorem \cite[Theorem 10.12.2]{schneider}, there exist $A\in GL(n,\R)$, $a\in \R^n$ and $r>0$ such that $B_2^n(a,r)\subset AK\subset B_2^n(a,nr)$. Since $AK$ contains the origin, $0\in AK\subset B_2^n(a,nr)$, $|0-a|= |a|<nr$. Thus, for any $x\in AK$, $|x|\leq |x-a|+ |a|\leq 2nr$, proving  (\ref{johneq1}). 

Since also $0\in (AK)^\circ$, by the same reasoning, there exists $\Tilde{r}>0$ and $b\in \R^n$ such that 
\begin{equation}\label{johneq3}
    B_2^n(b,\Tilde{r})\subset (AK)^\circ\subset B_2^n(b,n\Tilde{r})\subset B_2^n(0, 2n\Tilde{r}).
\end{equation}
To prove (\ref{johneq2}), note that $b(AK)=0$ by (\ref{affine_barycenter}). Thus, Remark \ref{remark_santalo} applies and using (\ref{johneq1}) and (\ref{johneq3}),
\begin{equation}\label{naz_eq300}
    |B_2^n(a,r)||B_2^n(b,\Tilde{r})|\leq \M(AK)= \inf_{z\in\R^n}\M\left((AK)^\circ-z\right) \leq |B_2^n(0,1)|^2.
\end{equation}
Note $|B_2^n(a,r)||B_2^n(b,\Tilde{r})|= r^n\Tilde{r}^n|B_2^n(0,1)|^2$. Thus, $r\tilde{r}\leq 1$ so (\ref{johneq2}) follows from (\ref{johneq3}).
\end{proof}

The next remark accompanies the proof of Lemma \ref{naz_john_lemma}. 
\begin{remark}\label{remark_santalo}
To justify (\ref{naz_eq300}), note that for a convex body $K$ with $b(K)=0$ the Santal\'o point of its polar is at the origin, $s(K^\circ)=0$, since $b((K^\circ-0)^\circ)= b(K)=0$ \cite[p. 157]{santalo}. As a result, by Santal\'o's inequality \cite[(3.12)]{santalo}
\begin{equation*}
    \M(K)= \M(K^\circ)= \inf_{z\in \R^n}\M(K^\circ-z)\leq |B_2^n(0,1)|^2.
\end{equation*}
\end{remark}

\begin{remark}
\label{avoidJohnRk}
One may avoid the use of John's theorem at the cost of slightly less precise estimates down the road. Yet, using tensorization one can still obtain from these Proposition \ref{nazprop2}. To see this, for a convex body $K\subset \R^n$ with $b(K)=0$ let $r,R>0$ such that 
\begin{equation*}
    B^n_2(0,r)\subset K\subset B^n_2(0,R). 
\end{equation*}
(Of course, $r$ and $R$ depend on $K$. Moreover, unlike in our previous work that did invoke 
John's theorem, $R/r$ depends on $K$ and may not be uniformly controlled.)
Now, $B_2^n(0, 1/R)\subset K^\circ\subset B^n_2(0,1/r)$, thus for any $x\in K$ and $t\in K^\circ$,
$
    |\langle x,t\rangle|\leq |x||t|\leq {R}/{r}, 
$
and hence
$
    \frac{r}{R\sqrt{2}}K\times K\subset K_\C\subset K\times K.
$
As a result, since $B_2^n(0,r)\subset K$, 
\begin{equation*}
   B^{2n}_2\left( 0, \frac{r^2}{R\sqrt{2}}\right)= \frac{r}{R\sqrt{2}} B^{2n}_2(0,r)\subset \frac{r}{R\sqrt{2}} B^n_2(0,r)\times B^n_2(0,r)\subset \frac{r}{R\sqrt{2}}K\times K\subset K_\C. 
\end{equation*}
As in Lemma \ref{distance_lemma} this shows that $\mathrm{dist}(\C^n\setminus (\sigma\delta K_\C), \delta K_\C)\geq \frac{(\sigma-1)\delta r^2}{R\sqrt{2}}$, which allows for a bump function $g:\C^n\to \R$ supported on $\sigma\delta K_\C$, equal to 1 in $\delta K_\C$ with 
$
    |\dif g\,|\leq \frac{2R\sqrt{2}}{(\sigma-1)\delta r^2}.
$
Therefore, 
$
    \int_{T_K}|\overline{\partial}g|^2 e^{-\phi}\leq \frac{2R^2}{(\sigma-1)^2\delta^2 r^4} e^{-2n\log\delta+ 2nC\delta} (\sigma\delta)^{2n} |K|^2,
$
and
$
    \int_{T_K}|h|^2 \leq \left( \frac{R}{r}\right)^4 \frac{8 e^{1+2nC\delta}}{(\sigma-1)^2 \delta^2} (4\sigma^2)^n |K|^2.
$
As a result, (\ref{nazeq100}) becomes 
$
    |K|^2\mathcal{K}_{T_K}(0,0)\geq \left( \frac{1}{4\sigma^2}\right)^n \left( \frac{r}{R}\right)^4 \frac{(\sigma-1)^2\delta^2}{8 e^{1+2nC\delta}}. 
$
Note that for $m\in\mathbb{N}$, 
\begin{equation*}
    B^{nm}_2(0,r)\subset B^n_2(0,r)\times\ldots\times B^n_2(0,r)\subset K^m\subset B^n_2(0,R)\times \ldots\times B^n_2(0,R)\subset B_2^{nm}(0, mR).
\end{equation*}
Therefore, 
\begin{equation}\label{nazeq99}
    \left( |K|^2 \mathcal{K}_{T_K}(0,0)\right)^m= |K^m|^2 \mathcal{K}_{T_{K^m}}(0,0)\geq \left( \frac{1}{4\sigma^2}\right)^{nm} \left(\frac{r}{mR} \right)^4 \frac{(\sigma-1)^2 \delta^2}{8 e^{1+2nm C\delta}}. 
\end{equation}
One recovers (\ref{nazeq100}) by first taking $m\to \infty$, then $\sigma\to 1$ and $\delta\to 0$.
\end{remark}

\subsection{Plurisubharmonic support function}
\label{pshsupp_sub}

In the symmetric setting, Nazarov invokes a neat trick
of using a kind of complexified (plurisubharmonic) support function
to construct the weight function. On the tube domain the inner product
$\langle z,t\rangle$ (with $z\in T_K$ and $t\in K^\circ$) 
is, of course, a complex number, lying in the strip
$\R+\i[-1,1]\subset\C$; 
trying to consider a naive candidate for a complexified support
function of the form $\log\sup_{t\in K^\circ}|\langle z,t\rangle|$
would not quite work as it is unbounded. To overcome this Nazarov
composes with the 
conformal map $\Phi_s(\zeta)= \frac{4}{\pi}\frac{e^{\frac{\pi}{2}\zeta}-1}{e^{\frac\pi2 \zeta}+1}$ 
sending the strip $\{|\mathrm{Im}\zeta|\leq 1\}$ to the disk of radius $\frac{4}{\pi}$ \cite[p. 339]{nazarov}, that leads to the
nicely bounded psh function
\begin{equation}
\label{usefulcxsupp}
\log\sup_{t\in K^\circ}|\Phi(\langle z,t\rangle)|
\end{equation}
(that Nazarov modifies by a quadratic term to obtain 
a desirable weight function).
There is a small caveat (not discussed in \cite{nazarov}) that we deal with in Lemma \ref{properties_list} (i): to show that this supremum in
\eqref{usefulcxsupp} is indeed plurisubharmonic, it is necessary to check it is upper semi-continuous. 

In the non-symmetric case, the expression
$\langle z,t\rangle$ (with $z\in T_K$ and $t\in K^\circ$) 
now lies in the half-space
$\R+\i[-\infty,1]\subset\C$; composing with 
the conformal map from the half-plane $H\defeq \{\zeta\in \C: \mathrm{Im}\zeta\leq 1\}$ to the closed disk of radius 2 (as suggested by Nazarov \cite[p.342]{nazarov}),
\begin{gather}
\label{PhiEq}
    \Phi: \{\mathrm{Im}\zeta\leq 1\}\to \overline{B_2^2(0, 2)}, \qquad
    \zeta\mapsto \frac{-2\sqrt{-1} \zeta}{\zeta-2\sqrt{-1}},
\end{gather}
one is again led to the useful complexified support function
\eqref{usefulcxsupp}.
A simple auxiliary estimate is needed for later calculations:

\begin{claim}\label{bound_on_Phi}
There exists $C>0$ such that $|\log|\Phi(\zeta)|-\log|\zeta||\leq C|\zeta|$, for all $|\zeta|\leq \frac12$.
\end{claim}

\begin{remark}
In Nazarov's symmetric setting he uses (without proof) the estimate 
$|\log|\Phi_s(\zeta)|- \log|\zeta||\leq C|\zeta|$ when $|\zeta|\leq 1/2$  \cite[p. 339]{nazarov}.
\end{remark}

\begin{proof}
For
\begin{equation*}
    \Psi(\zeta)\defeq \frac{-2\sqrt{-1}}{\zeta-2\sqrt{-1}}= \begin{cases} \frac{\Phi(\zeta)}{\zeta},  &\zeta\neq 0, \\
    1,  &\zeta =0,
    \end{cases}
\end{equation*}
compute,
$
    \Psi'(\zeta)= \frac{2\sqrt{-1}}{(\zeta-2\sqrt{-1})^2}.
$
In particular, $\Psi(0)= 1$ and $\Psi'(0)= -\frac{\sqrt{-1}}{2}$. Since $\Psi(\zeta)\neq 0$ for $|\zeta|\leq \frac12$, $\log\Psi(\zeta)$ is well-defined and holomorphic. Note, 
\begin{equation*}
\log\Psi(0)=0, \quad 
    (\log\Psi)'(0)= \frac{\Psi'(0)}{\Psi(0)}= -\frac{\sqrt{-1}}{2},
\end{equation*}
so there exists $g(\zeta)$ holomorphic in $B_2^2(0,\frac12)$ with $g(0)=-\frac{\sqrt{-1}}{2}$ so that $\log\Psi(\zeta)= \zeta g(\zeta)$. Take $C= \sup_{|\zeta|\leq \frac12}|g(\zeta)|$. By the maximum principle, $C\geq |g(0)|= 1/2$. Since $\log|\Psi|= \mathrm{Re}\log\Psi$,
\begin{equation*}
\left| \log|\Psi(\zeta)| \right|\leq |\log\Psi(\zeta)|= |\zeta||g(\zeta)|\leq C|\zeta|,
\end{equation*}
for all $|\zeta|\leq \frac12$.
\end{proof}

\begin{remark}
There is flexibility in the choice of the conformal map. For instance, replacing \eqref{PhiEq}
by a map $\Phi$ from $\{\mathrm{Im}\zeta\leq 1\}$ to the unit ball $\overline{B_2^2(0,1)}$
would give $\Phi'(0)= 1/2$ and the estimate of Claim \ref{bound_on_Phi} would change to
 $   |\log|\Phi(\zeta)|- \log|\zeta/2||\leq C|\zeta|
 $.
As the reader may readily verify,
the only substantial subsequent changes would be in Lemma \ref{properties_list} (ii) and (iv) which would read $\phi(z)\leq 1, z\in T_K$, and 
\begin{equation*}
    \phi(z)\geq 2n\log(\delta/2)- 16 C\delta n^3\sqrt{2}, \quad z\in (\sigma\delta K_\C- (\sigma-1)\delta \tilde{a})\setminus \delta K_\C,
\end{equation*}
respectively. In any case, the estimate on the $L^2$-norm of $h$ (\ref{h_estimate2}) would remain the same since
\begin{equation}\label{naz_eq502}
        \int_{T_K}|h|^2\dif\lambda\leq e^{\sup_{z\in T_K} \phi(z)} \int_{T_K} |h|^2 e^{-\phi}\dif\lambda, 
\end{equation}
is bounded above by the product of $e^{\sup_{T_K}\phi(z)}$ and the weighted $L^2$-norm of $h$. Changing $\Phi$ would result in a trade-off between estimates in those two terms. First of all, in this case
\begin{equation*}
    e^{\sup_{z\in T_K}\phi(z)}\leq e, 
\end{equation*}
instead of $e^{1+2n\log 2}= 4^n e$. On the other hand, (\ref{g_estimate_eq}) would become
\begin{equation*}
    \begin{aligned}
        \int_{T_K}|\overline{\partial}g|^2e^{-\phi}\dif\lambda 
        &\leq \frac{32n^4}{(\sigma-1)^2\delta^2 r^2} (\delta/2)^{-2n}\sigma^{-2n} e^{2n\log\sigma+ 16n^3C\delta\sqrt{2}} (\sigma\delta)^{2n}|K|^2= (4\sigma^2)^n e^{o(n)} |K|^2,
    \end{aligned}
\end{equation*}
i.e., $(\delta)^{-2n}$ is replaced by $(\delta/2)^{-2n}$, thus (\ref{h_estimate}) becomes
\begin{equation*}
    \int_{T_K} |h|^2 e^{-\phi}\dif\lambda\leq  8n^2 r^2 \int_{T_K} |\overline{\partial}g|^2 e^{-\phi}\dif \lambda\leq (4\sigma^2)^n e^{o(n)}|K|^2, 
\end{equation*}
i.e., $\sigma^{2n}$ is replaced by $4^n\sigma^{2n}$, resulting in the same estimate in (\ref{naz_eq502}).
\end{remark}

\subsection{Weight function}
\label{weight_function_sub}
At this point an important reduction is needed.
In order to prove Proposition \ref{nazprop2},
  by replacing $K$ with $A(K-b(K))$
 and 
invoking
 Lemmas \ref{bk_ai} and \ref{naz_john_lemma}, 
it is enough to restrict to a sub-class of ``John convex bodies":
\begin{equation*}
\begin{aligned}
    \mathcal{J}\defeq \{K\subset \R^n &\text{ convex body with } b(K)=0 \text{ such that the} \\&\text{ conclusion of Lemma \ref{naz_john_lemma} holds with } A=I_n\}.
\end{aligned}
\end{equation*}
Thus, for the remainder of the Section, fix $K\in \mathcal{J}$.

For $K\in\mathcal{J}$, let $r>0$ and $a\in \R^n$ be such that (\ref{johneq1}) and (\ref{johneq2}) hold.
Building on \eqref{usefulcxsupp}, consider the plurisubharmonic weight function on $T_K$,
\begin{equation}\label{weight_phi}
    \phi(z)\defeq \frac{|\mathrm{Im}\,z|^2}{4n^2r^2}+ 2n\log\sup_{t\in K^\circ}|\Phi(\langle z,t\rangle)|.
\end{equation}
(The advantage of using $|\mathrm{Im}\,z|^2$, and not $|z|^2$, is that the resulting
$\phi$ is bounded.)
The main properties of $\phi$ are the content of
the next lemma
that requires two more
pieces of notation.
Set 
\begin{equation}\label{KC}
    K_\C\defeq \{z\in \C^n: |\langle z, t\rangle|\defeq\sqrt{\langle x,t\rangle^2+ \langle y,t\rangle^2}\leq 1 \text{ for all } t\in K^\circ\},
\end{equation}
 and
 \begin{equation}
 \label{TildeaEq}
 \Tilde{a}\defeq a+\sqrt{-1}a\in T_K. 
 \end{equation}

\begin{lemma}\label{properties_list}
Let $K\in\mathcal{J}$. Let $\phi$ 
and $K_\C$ be given by (\ref{weight_phi}) and (\ref{KC}).

\noindent 
    (i) $\phi$ is plurisubharmonic and satisfies the conditions of Proposition
    \ref{hormander} with $\tau= 1/(8n^2r^2)$. 

\noindent     
    (ii) For all $z\in T_K$, $\phi(z)\leq 2n\log 2+1$.

\noindent 
    (iii) 
    $e^{-\phi}$ is not locally integrable at $0$: 
    \begin{equation*}
        e^{-\phi(z)}\geq e^{-1-nC}\frac{r^{2n}}{(2n)^{2n}|z|^{2n}},
        \quad\hbox{for $z\in T_K$ with $|z|\leq(4n)^{-1}r$}. 
    \end{equation*}

\noindent 
    (iv) For $\sigma\in (1,2)$, $\delta\in (0,\frac{1}{8(1+2\sqrt{2}n^2)})$, 
    $a\in\R^n$ as in Lemma \ref{naz_john_lemma}, and $\Tilde{a}\in\C^n$ as in \eqref{TildeaEq}, 
    \begin{equation*}
        \phi(z)\geq 2n\log\delta- 16\sqrt{2}C\delta n^3, \quad z\in (\sigma\delta K_\C-(\sigma-1)\delta \Tilde{a})\setminus \delta K_\C,
    \end{equation*}
and the set $(\sigma\delta K_\C-(\sigma-1)\delta \Tilde{a})\setminus \delta K_\C$
is non-empty.

\noindent     
    (v) $\frac{1}{4\sqrt{2}n^2}(K\times K)\subset K_\C\subset K\times K.$
\end{lemma}

\begin{remark}
The lemma is a little different from Nazarov's bounds for symmetric bodies \cite[\S5--\S 6]{nazarov}. The most substantial differences are in Lemma \ref{properties_list} (iv), for which $\delta$ needs to move in a smaller range for Claim \ref{bound_on_Phi} to hold, and the bound is now for $z\in(\sigma\delta K_\C-(\sigma-1)\delta \Tilde{a})\setminus \delta K_\C$ instead of $z\in\sigma\delta K_\C\setminus\delta K_\C$. Moreover, for Lemma \ref{properties_list} (v) a smaller dilation, dependent on dimension, is necessary for the first inclusion to hold. This is due to Lemma \ref{naz_john_lemma}, which gives
  $  |\langle x,t\rangle|\leq 4n^2,$ 
for all $x\in K, t\in K^\circ$ (in place of $|\langle x,t\rangle|\leq 1$ for symmetric bodies).
\end{remark}

\begin{proof}
\noindent 
    (i)
    The plurisubharmonicity of $\log\sup_{t\in K^\circ}|\Phi(\langle z,t\rangle)|$ is proved as follows. First, note that $\Phi(\langle z,t\rangle)$ is holomorphic in $z$ for all $t\in K^\circ$, thus $\log|\Phi(\langle z,t\rangle)|$ is plurisubharmonic on $z$ \cite[Example 4.1.10]{hormander}, and so $\sup_{t\in K^\circ}\log|\Phi(\langle z,t\rangle)|=\log\sup_{t\in K^\circ}|\Phi(\langle z,t\rangle)|$ is plurisubharmonic if it is upper semi-continuous \cite[Theorem 5]{lelong} (here we used $\log$ is increasing). Since $\log$ is increasing, it suffices to show that $\sup_{t\in K^\circ}|\Phi(\langle z,t\rangle)|$ is upper semi-continuous. 
    
    Fix $z_0\in T_K$ and let $\rho>0$ such that $\overline{B_2^{2n}(z_0, \rho)}\subset T_K$. Note that for $z\in \overline{B_2^{2n}(z_0, \rho)}$ and $t\in K^\circ$, by (\ref{johneq2}), $|\langle z,t\rangle|\leq (\rho+|z_0|)2n r^{-1}$. On $\{\mathrm{Im}\zeta\leq 1\}\cap \overline{B_2^2(0, (\rho+|z_0|) 2n r^{-1})}$, since it is compact and $\Phi$ is continuous, $\Phi$ is uniformly continuous. Let $\e>0$. There exists $\delta>0$ such that for $\zeta_1, \zeta_2\in \{\mathrm{Im}\zeta\leq 1\}\cap\overline{B_2^2(0,(\rho+|z_0|) 2n r^{-1})}$ with $|\zeta_1-\zeta_2|<\delta$, $|\Phi(\zeta_1)-\Phi(\zeta_2)|<\e$. For $z\in \overline{B_2^{2n}(z_0, \rho)}$ with $|z-z_0|<\frac{r\delta}{2n}$, and $t\in K^\circ$, by (\ref{johneq2}),
    \begin{equation*}
        |\langle z,t\rangle- \langle z_0, t\rangle|= |\langle z-z_0, t\rangle|\leq |z-z_0||t|<\frac{r\delta}{2n}\frac{2n}{r}= \delta, 
    \end{equation*}
    thus $|\Phi(\langle z,t\rangle)- \Phi(\langle z_0, t\rangle)|<\e$, 
    i.e., $|\Phi(z,t)|\leq \e+ |\Phi(\langle z_0, t\rangle)|$. Taking supremum over $t\in K^\circ$, 
    \begin{equation*}
        \sup_{t\in K^\circ}|\Phi(\langle z,t\rangle)|\leq \e+ \sup_{t\in K^\circ}|\Phi(z_0, t\rangle)|,
    \end{equation*}
    from which the upper semi-continuity of $\sup_{t\in K^\circ}|\Phi(\langle z,t\rangle)|$ follows.     
    
    Given that $\log \sup_{t\in K^\circ}|\Phi(\langle z,t\rangle)|$ is plurisubharmonic, compute 
    \begin{equation*}
        \begin{aligned}
            \sqrt{-1}\partial\overline{\partial}\phi&\geq \frac{\sqrt{-1}}{4n^2r^2}\partial\overline{\partial}|\mathrm{Im}\,z|^2= \frac{\sqrt{-1}}{4n^2 r^2}\partial\overline{\partial}\left|\frac{z-\overline{z}}{2\sqrt{-1}}\right|^2=\frac{\sqrt{-1}}{4n^2 r^2}\partial\overline{\partial}\frac{z\overline{z}}{2}= \frac{\sqrt{-1}}{8n^2r^2}\sum_{j=1}^n \dif z_j\wedge\dif \overline{z}_j.
        \end{aligned}
    \end{equation*}

\noindent 
    (ii) $\Phi$ maps $\{\zeta\in\C: \mathrm{Im}\,\zeta\leq 1\}$ to $\overline{B_2^2(0,2)}$, thus $|\Phi(\langle z,t\rangle)|\leq 2$ for all $z\in T_K$ and $t\in K^\circ$. Moreover, for $z=x+\sqrt{-1}y\in T_K$, $y\in K$ so by (\ref{johneq1}), 
    \begin{equation}\label{naz_eq400}
    |\mathrm{Im}\, z|=|y|\leq 2nr,
    \end{equation}
    thus $\phi(z)\leq 1+ 2n\log 2$. 
    
\noindent     
    (iii) By Claim \ref{bound_on_Phi} there exists $C>0$ such that  $|\log|\Phi(\zeta)|-\log|\zeta||\leq C|\zeta|$, for all $|\zeta|\leq \frac12$. 
    By (\ref{johneq2}), $|t|\leq 2n r^{-1}$, for all $t\in K^\circ$. By Cauchy--Schwarz, if $z\in T_K$ such that $|z|\leq (4n)^{-1}r$, the inner product $|\langle z,t\rangle|\leq 1/2$ thus, 
    \begin{equation*}
    \begin{aligned}
        \log|\Phi(\langle z,t\rangle)|\leq C|\langle z,t\rangle|+ \log|\langle z,t\rangle|&\leq \frac{C}{2}+ \log(|z||t|)  
        \\&\leq \frac{C}{2}+ \log\frac{2n|z|}{r}, \raisebox{-0\normalbaselineskip}{ \quad $z\in T_K \text{ with } |z|\leq r/(4n), t\in K^\circ$}.
    \end{aligned}
    \end{equation*}
    Thus,
    \begin{equation*}
        2n \log\sup_{t\in K^\circ}|\Phi(\langle z,t\rangle)|\leq nC+ 2n\log\frac{2n|z|}{r}= nC+ \log\frac{(2n)^{2n} |z|^{2n}}{r^{2n}}, \quad z\in T_K \text{ with }|z|\leq r/(4n).
    \end{equation*}
    Therefore, by (\ref{naz_eq400}) and (\ref{weight_phi}),
   \begin{equation*}
        e^{-\phi(z)}\geq e^{-1-nC}\frac{r^{2n}}{(2n)^{2n}|z|^{2n}}, \quad  \text{ for all } z\in T_K \text{ with } |z|\leq r/(4n).
    \end{equation*}

\noindent  
(iv) Let $z\in (\sigma\delta K_\C- (\sigma-1)\delta\Tilde{a})$, i.e., $|\langle z+ (\sigma-1)\delta \Tilde{a}, 
t\rangle|\leq \sigma\delta$, for all $t\in K^\circ$. By the triangle inequality,
\begin{equation}\label{list_lemma_eq1}
    |\langle z,t\rangle|- (\sigma-1)\delta |\langle \Tilde{a},t\rangle|\leq 
    \sigma\delta, \quad \text{ for all } t\in K^\circ. 
\end{equation}
Note that $\langle \Tilde{a},t\rangle= \langle a,t\rangle+ \sqrt{-1}\langle a,t\rangle$, thus, by (\ref{johneq1}), (\ref{johneq2}), since $a\in K$, 
\begin{equation}\label{list_lemma_eq2}
    |\langle \Tilde{a},t\rangle|= |\langle a,t\rangle|\sqrt{2}\leq |a||t|\sqrt{2}\leq  2nr\frac{2n}{r}\sqrt{2}= 4\sqrt{2}n^2 ,
\end{equation}
for all $t\in K^\circ$. Combining (\ref{list_lemma_eq1}), (\ref{list_lemma_eq2}), 
\begin{equation}\label{list_lemma_eq3}
    |\langle z,t\rangle|\leq \sigma\delta+ (\sigma-1)\delta 4\sqrt{2}n^2\leq 2\delta
    + 4\sqrt{2}\delta n^2, \quad \text{ for all } t\in K^\circ,
\end{equation}
since $\sigma<2$. Note that for $\delta\in (0, \frac{1}{8(1+2\sqrt{2}n^2)})$ 
the right-hand side in (\ref{list_lemma_eq3}) is less than or equal to $1/2$, thus Claim \ref{bound_on_Phi} applies.
Next, assume that in addition also $z\notin (\delta K_\C)$,
i.e., $z\in (\sigma\delta K_\C- (\sigma-1)\delta\Tilde{a})\setminus \delta K_\C$.
In particular, there exists $t_0\in K^\circ$ such that $|\langle z,t_0\rangle|> \delta$.
As a result, by Claim \ref{bound_on_Phi} and (\ref{list_lemma_eq3}),
\begin{equation*}
    \log|\Phi(\langle z,t_0)\rangle|\geq \log|\langle z,t_0\rangle|- C|\langle z,t_0\rangle|
    \geq 
    \log\delta- 2C\delta- 4\sqrt{2}C\delta n^2
    \ge \log\delta- 8\sqrt{2}C\delta n^2,
\end{equation*}
because $2C\delta\leq 4\sqrt{2}C\delta n^2$.
From the definition \eqref{weight_phi} then,
\begin{equation*}
    \phi(z)\geq 2n\log|\Phi(z, t_0)|\geq 2n\log\delta- 16\sqrt{2}C\delta n^3.
\end{equation*}
Finally, the set $(\sigma\delta K_\C-(\sigma-1)\delta \Tilde{a})\setminus \delta K_\C$
is non-empty by Claim \ref{nonemptyClaim} below.

\noindent 
(v) For $z= x+\sqrt{-1}y\in K_\C$, $\langle x,t\rangle^2+ \langle y,t\rangle^2\leq 1$, for all $t\in K^\circ$, thus $\langle x,t\rangle\leq 1$ and $\langle y,t\rangle\leq 1$, for all $t\in K^\circ$, i.e., $x\in K$ and $y\in K$.

For the second inclusion, note that 
by Cauchy--Schwarz and (\ref{johneq1}), (\ref{johneq2}),
\begin{equation*}
    1\geq\langle x,t\rangle\geq -|x||t|\geq -2nr \frac{2n}{r}= -4n^2,
\end{equation*}
thus, $|\langle x,t\rangle|\leq 4n^2$, and similarly $|\langle y,t\rangle|\leq 4n^2$, for all $x,y \in K, t\in K^\circ$. As a result, for $x,y\in (4n^2\sqrt{2})^{-1} K$, 
\begin{equation*}
    \langle x,t\rangle^2+ \langle y,t\rangle^2\leq \left(\frac{1}{\sqrt{2}}\right)^2+ \left(\frac{1}{\sqrt{2}}\right)^2=1,
\end{equation*}
that is, $z=x+\sqrt{-1}y\in K_\C$.
\end{proof}

\begin{claim}
\label{nonemptyClaim}
The set $(\sigma\delta K_\C-(\sigma-1)\delta \tilde{a})\setminus \delta K_\C$ 
in Lemma \ref{properties_list} (iv) is non-empty. 
\end{claim}
\begin{proof}
This follows from Lemma \ref{distance_lemma} below. Indeed, for any two  non-empty  sets $S_1, S_2\subset \R^k$ with $S_1\setminus S_2=\emptyset$, one has $S_1\subset S_2$
and 
$0\leq \mathrm{dist}(\R^k\setminus S_1, S_2) \leq \mathrm{dist}(\R^k\setminus S_2, S_2)=0.
$
\end{proof}

\subsection{Bump function}\label{bump_function_sub}
For symmetric bodies, one 
needs only an easier estimate on the distance between $\C^n\setminus (\sigma\delta K_\C)$ and $\delta K_\C$ since $\tilde{a}= 0\in\C^n$
(recall \eqref{TildeaEq}), where $a\in \mathrm{int}\,K$ as in (\ref{johneq1}) \cite[Lemma 16]{hultgren}.
For the general setting the following lower bound on the distance between $\C^n\setminus (\sigma\delta K_\C-(\sigma-1)\delta\Tilde{a})$ and $\delta K_\C$ holds.
Aside from being useful for the construction of the bump function below, the next
lemma is also needed for one of the properties of the weight function above
(see Claim \ref{nonemptyClaim}):

\begin{lemma}\label{distance_lemma}
Let $K\in\mathcal{J}$. For $\sigma>1$, $\delta>0$, $K_\C$ as in (\ref{KC}) and $\tilde{a}$ as in (\ref{TildeaEq}), 
\begin{equation*}
    \mathrm{dist}\big(\C^n\setminus (\sigma\delta K_\C-(\sigma-1)\delta \Tilde{a}), \delta K_\C\big)\geq \frac{(\sigma-1)\delta r}{4n^2\sqrt{2}}.
\end{equation*}
\end{lemma}
\begin{proof}
By (\ref{johneq1}), $B_2^n(a,r)\subset K$, so
\begin{equation*}
    B_2^{2n}\Big(\Tilde{a}, \frac{r}{4n^2\sqrt{2}}\Big)
    \subset 
    B_2^n\Big(a, \frac{r}{4n^2\sqrt{2}}\Big)
    \times 
    B_2^{n}\Big(a, \frac{r}{4n^2\sqrt{2}}\Big)
    \subset \frac{K\times K}{4n^2\sqrt{2}}\subset K_\C,
\end{equation*}
by Lemma \ref{properties_list} (v). As a result, 
\begin{equation*}
\begin{aligned}
    \delta K_\C+ (\sigma-1)\delta B_n^{2n}(0, \frac{r}{4n^2\sqrt{2}})&= \delta K_\C+ (\sigma-1)\delta B_n^{2n}(\Tilde{a},\frac{r}{4n^2\sqrt{2}})- (\sigma-1)\delta \Tilde{a}\\
    &\subset \delta K_\C+ (\sigma-1)\delta K_\C- (\sigma-1)\delta \Tilde{a}.\\
\end{aligned}
\end{equation*}
To conclude the proof, observe that $\delta K_\C+ (\sigma-1)\delta K_\C=\sigma\delta K_\C$:
$\sigma\delta K_\C\subset \delta K_\C+ (\sigma-1)\delta K_\C$
since $\sigma\delta x= \delta x+ (\sigma-1)\delta x$;
conversely, for $\delta x\in \delta K_\C$ and $(\sigma-1)\delta y\in (\sigma-1)\delta K_\C$,
\begin{equation*}
    \delta x+ (\sigma-1)\delta y= \sigma\delta\left(\frac{1}{\sigma}x+ \frac{\sigma-1}{\sigma}y \right)\in \sigma\delta K_\C,
\end{equation*}
by convexity. 
\end{proof}

\begin{remark}
Centering at $(\sigma-1)\delta \tilde{a}$ comes from John's theorem
(Lemma \ref{naz_john_lemma}).
Lemma \ref{distance_lemma} explains why it is necessary to consider $\sigma\delta K_\C-(\sigma-1)\delta \tilde{a}$ instead of $\sigma \delta K_\C$.  For the latter, since $0\in\mathrm{int}\,K$, one can still find a ball centered at the origin and contained in $K$ but with possibly too small of a radius to provide a substantial lower bound on the distance between $\sigma\delta K_\C$ and $\delta K_\C$.
\end{remark}

Next, we show that
Lemma \ref{distance_lemma} allows for a bump function $g$ with good enough control on the derivative (though dependent of $n$), supported on $\sigma\delta K_\C-(\sigma-1)\delta\Tilde{a}$ and equal to $1$ in $\delta K_\C$
(in the symmetric setting, the bump function
is supported on $\sigma\delta K_\C$ and its gradient is bounded from above
independent of the dimension \cite[p. 340]{nazarov}).
Before stating the result we define some notation
concerning 1-forms.
For a function $f$ on $T_K$, $df=\partial f+\bar\partial f$, with $\partial f=\sum_{i=1}^n\frac{\partial f}{\partial z_i}dz_i$ and
$\bar\partial f=\sum_{i=1}^n\frac{\partial f}{\partial \bar{z_i}}\overline{dz_i}$.
For a $(0,1)$-form $\omega= \sum_j a_j \dif\overline{z_j}$, set
\begin{equation*}
|\omega|^2:= \sum_j |a_j|^2, 
\end{equation*}
and similarly for $(1,0)$-forms. 
Also,
$|\dif f|^2= |\partial{f}|^2+ |\overline{\partial f}|^2$. If, in addition, $f$ only takes real values, $\partial f/\partial \overline{z_i}= \overline{\partial f/ \partial z_i}$ thus $\overline{\partial} f= \overline{\partial f}$ and hence 
\begin{equation}\label{real_values}
|\dif f\,|^2= 2|\overline{\partial}f|^2, \quad \text{ for } f: \C^n\to \R. 
\end{equation}

\begin{lemma}\label{controlled_bump_function}
Let $K\in\mathcal{J}$
and
let $K_\C$ be given by (\ref{KC}).
For $\sigma>1, \delta>0$, there exists a smooth $g:\C^n\to [0,1]$, supported on 
$\sigma\delta K_\C- (\sigma-1)\delta\Tilde{a}$ such that $g=1$ on $\delta K_\C$ and 

\begin{equation*}
    |\dif g\,|\leq \frac{8n^2\sqrt{2}}{(\sigma-1)\delta r}. 
\end{equation*}
\end{lemma}
\begin{proof}
Denote by $R\defeq ((\sigma-1)\delta r)/(4n^2\sqrt{2})$. By Lemma \ref{distance_lemma}, 
 $   \delta K_\C+ B_2^n(0, R)\subset \sigma\delta K_\C -(\sigma-1)\delta \Tilde{a}.
$
Consider the map 
\begin{equation*}
    G(x)= \begin{cases}
    1, \text{ if } d_{\delta K_\C}(x)\leq \frac{R}{4}, \\
    2- \frac{2}{R} (d_{\delta K_\C}(x)+\frac{R}{4}), \text{ if } \frac{R}{4}\leq d_{\delta K_\C}(x)\leq \frac{3R}{4} \\
    0, \text{ otherwise}, 
    \end{cases}
\end{equation*}
where the distance function of a non-empty set $S\subset \R^k$ is defined
by $d_S(x)\defeq \inf_{y\in S}|x-y|$. 
Thus, $G$
is continuous, equal to 1 in $\delta K_\C$, and 0 outside $(\sigma\delta K_\C-(\sigma-1)\delta \Tilde{a})$. Let $\chi$ a smooth, non-negative function, compactly supported on $B_2^{2n}(0, \frac{R}{4})$,
with $\int_{\R^{2n}} \chi\dif\lambda =1$ 
The convolution, 
\begin{equation*}
    g(x)\defeq (G\ast \chi)(x), 
\end{equation*}
is smooth, $g= 1$ in $\delta K_\C$, and $g= 0$ outside $(\sigma\delta K_\C-(\sigma-1)\delta \Tilde{a})$.
For $S\subset \R^k$ a submanifold of codimension at least 1, $\dif\, (d_S)$ is almost everywhere defined  with $|\dif\, (d_S)\,|=1$ 
\cite[Lemma 1.5.9]{schneider}, \cite[Example 22]{petersen}. 
Hence,
\begin{equation*}
    |\dif g\,|= |\dif G\ast \chi|\leq \frac{2}{R}= \frac{8n^2\sqrt{2}}{(\sigma-1)\delta r}, 
\end{equation*}
as claimed.
\end{proof}

By Lemmas \ref{properties_list} and \ref{controlled_bump_function} the following estimate on the weighted $L^2(T_K)$ norm of derivative of the bump function $\overline{\partial}g$ holds.
\begin{lemma}\label{g_estimate}
Let $K\in\mathcal{J}$
and
let $K_\C$ be given by (\ref{KC}).
For
$g$ as in Lemma \ref{controlled_bump_function} and $\phi$ as in (\ref{weight_phi}), $\sigma\in (1,2)$ and $\delta\in (0, \frac{1}{8(1+2n^2\sqrt{2})})$, 
\begin{equation}\label{g_estimate_eq}
\begin{aligned}
    \int_{T_K}|\overline{\partial}g|^2e^{-\phi(z)}\dif\lambda(z)&\leq \frac{16n^4}{(\sigma-1)^2\delta^2 r^2} e^{2n\log\sigma+ 16n^3C\delta\sqrt{2}} |K_\C|
    \\&\leq \frac{16n^4}{(\sigma-1)^2\delta^2 r^2} e^{2n\log\sigma+ 16n^3C\delta\sqrt{2}} |K|^2.
\end{aligned}
\end{equation}
\end{lemma}
\begin{proof}
Since $g$ is a real-valued function, by (\ref{real_values}), $|\overline{\partial}g|^2= |\dif g\,|^2/2$. 
Moreover, $\overline{\partial}g=0$ outside $(\sigma\delta K_\C- (\sigma-1)\delta \Tilde{a})$ and in $\delta K_\C$ since $g$ is constant there. As a result, by Lemma \ref{controlled_bump_function} and Lemma \ref{properties_list} (iv),
\begin{align*}
    \int_{T_K}|\overline{\partial}g|^2e^{-\phi}\dif\lambda &= \int_{T_K\cap ((\sigma\delta K_\C- (\sigma-1)\delta \Tilde{a})\setminus (\delta K_\C))} |\overline{\partial}g|^2 e^{-\phi} \dif\lambda\\ 
    &\leq \int_{(\sigma\delta K_\C- (\sigma-1)\delta\Tilde{a})\setminus \delta K_\C} \frac{16n^4}{(\sigma-1)^2\delta^2 r^2} e^{-2n\log\delta +16C\delta n^3 \sqrt{2}}\dif\lambda(z)\\
    &\leq \frac{16n^4}{(\sigma-1)^2\delta^2 r^2} \delta^{-2n} e^{16n^3C\delta\sqrt{2}}  |\sigma \delta K_\C|\\
    &\leq \frac{16n^4}{(\sigma-1)^2\delta^2 r^2} \delta^{-2n}\sigma^{-2n} e^{2n\log\sigma+ 16n^3C\delta\sqrt{2}} (\sigma\delta)^{2n}|K|^2,  
\end{align*}
because $|K_\C|\leq |K\times K|= |K|^2$ (Lemma \ref{properties_list} (v)).
\end{proof}

\subsection{Constructing the holomorphic function}
\label{lower_bound_sub}

The weight function $\phi$ (\ref{weight_phi}) and 
$\omega= \overline{\partial}g$ with $g$ the bump function
of Lemma \ref{controlled_bump_function} will next
be used in conjuction with H\"ormander's theorem
(Proposition \ref{hormander}). 

For a function $f$ defined on positive integers, 
\begin{equation}\label{little_o}
f(k)= o(k)\quad \hbox{ if \ \ }
    \lim_{k\to \infty}\frac{f(k)}{k}=0.
\end{equation}

\begin{lemma}\label{affine_nazprop2}
For $K\in\mathcal{J}$, $\mathcal{B}(K)\geq e^{o(n)} 4^{-n}$.
 \end{lemma}

\begin{proof}
Let $a\in \R^n$ and $r>0$ as in Lemma \ref{naz_john_lemma}.
Let also $\Tilde{a}$ as in (\ref{TildeaEq}), $\sigma\in(1,2)$ and $\delta\defeq 1/(16n^3)$. The conditions of Lemma \ref{g_estimate} are satisfied, so there exists $g$ supported on $(\sigma\delta K_\C- (\sigma-1)\delta \Tilde{a})$ with $g=1$ on $\delta K_\C$ and (\ref{g_estimate_eq}) holds. By Proposition \ref{hormander} and Lemma \ref{properties_list} (i), there exists $h:T_K\to \C$ solving 
\begin{equation}
\label{ghEq}
\overline{\partial}h= -\overline{\partial}g, 
\end{equation}
so that 
$
    \int_{T_K}|h|^2 e^{-\phi}\dif\lambda\leq 8n^2 r^2\int_{T_K}|\overline{\partial}g|^2 e^{-\phi}\dif \lambda. 
$
Therefore, by (\ref{g_estimate_eq}), 
\begin{equation}\label{h_estimate}
    \int_{T_K} |h|^2 e^{-\phi}\dif\lambda\leq  \frac{2^{15} n^{12}}{(\sigma-1)^2} e^{2n\log\sigma + C\sqrt{2}} |K|^2,
\end{equation}
because $\delta= 1/(2^4n^3)$. Moreover, for $z\in T_K$, by Lemma \ref{properties_list} (ii),
$\phi(z)\leq 1+ 2n\log 2$, thus by (\ref{h_estimate}),
\begin{equation}\label{h_estimate2}
    \int_{T_K}|h|^2\dif\lambda\leq e^{1+ 2n\log 2}\int_{T_K} |h|^2 e^{-\phi}\dif\lambda\leq 4^n \frac{2^{15}n^{12}}{(\sigma-1)^2}\sigma^{2n} e^{1+C\sqrt{2}}|K|^2.
\end{equation}
In particular, $\delta<\frac14$ so the product $\sigma\delta<1$.
Also, by Lemma \ref{controlled_bump_function},
$0\le g\le 1$ and is supported on 
$(\sigma\delta K_\C- (\sigma-1)\delta\Tilde{a})$.
Thus,
\begin{equation}\label{g_estimate2}
    \int_{T_K}|g|^2\dif\lambda \leq |(\sigma\delta K_\C- (\sigma-1)\delta\Tilde{a})|= (\sigma\delta)^{2n}|K_\C|\leq |K|^2,
\end{equation}
because $K_\C\subset K\times K$ (Lemma \ref{properties_list} (v)), thus $|K_\C|\leq |K|^2$.

For $z\in T_K$ consider the holomorphic (by \eqref{ghEq}) function 
$$
f(z)\defeq g(z)+h(z).
$$
We claim $f\in A^2(T_K)$. To see that, $|f|^2= |g+h|^2\leq 2|g|^2+2|h|^2$, hence by (\ref{h_estimate2}) and (\ref{g_estimate2}), since the right-hand side in (\ref{h_estimate2}) is bigger than the right-hand side in (\ref{g_estimate2}),
\begin{equation}\label{f_estimate}
    \|f\|_{L^2(T_K)}^2\defeq \int_{T_K}|f|^2\dif \lambda\leq 2\int_{T_K}|g|^2\dif\lambda+ 2\int_{T_K}|h|^2\dif\lambda\leq 4^n \frac{2^{17}n^{12}}{(\sigma-1)^2}\sigma^{2n} e^{1+C\sqrt{2}}|K|^2. 
\end{equation}
On the other hand, by Lemma \ref{properties_list}, $e^{-\phi}$ is comparable to $|z|^{-2n}$ around the origin, which is not integrable. But by (\ref{h_estimate}), $|h|^2e^{-\phi}$ is integrable in $T_K$, thus $h(0)=0$. Furthermore, by 
construction (Lemma \ref{controlled_bump_function}) $g(0)=1$, thus $f(0)= g(0)+ h(0)=1$. By (\ref{f_estimate}) and Lemma \ref{equivalent_norms}, 
\begin{equation}\label{nazeq100}
      \mathcal{K}_{T_K}(0,0)\geq \frac{|f(0)|^2}{\|f\|_{L^2(T_K)}^2}\geq  \frac{(\sigma-1)^2 e^{-1-C\sqrt{2}}}{(4\sigma^2)^n 2^{18}n^{12} |K|^2}= e^{o(n)} \frac{1}{(4\sigma^2)^n |K|^2}.
\end{equation}
Taking $\sigma\to 1$ proves the lemma.
\end{proof}

Since $K$ is convex, $T_K$ is also convex, and hence a pseudoconvex domain \cite[Theorem 4.2.8, Corollary 2.5.6]{hormander3}, and H\"ormander's Theorem
applies \cite[Theorem 2.2.1']{hormander2}:
\begin{proposition}\label{hormander}
Suppose $\phi$ is plurisubharmonic and
 $   \sqrt{-1}\partial\overline{\partial}\phi\geq \tau \sum_{i=1}^n \sqrt{-1}\dif z_i\wedge\dif\overline{z}_i
$
 on $T_K$
for some constant $\tau>0$. Then, for any (0,1)-form $\omega$ with $\overline{\partial}\omega=0$, there is $h$ in $T_K$ such that $\overline{\partial}h=\omega$, with
\begin{equation*}
    \int_{T_K} |h|^2 e^{-\phi}\dif\lambda(z)\leq \tau^{-1}\int_{T_K}|\omega|^2e^{-\phi}\dif\lambda(z).
\end{equation*}
\end{proposition}

\subsection{Tensorization and Bergman kernels}
\label{tensorSec}

The Bergman kernel of the Cartesian product is the product of the Bergman kernels.
\begin{lemma}\label{bk_product}
For $\Omega_1\subset \C^n, \Omega_2\subset \C^m$ domains, and $z_1,w_1\in \Omega_1,
z_2,w_2\in \Omega_2$, 
\begin{equation*}
    \mathcal{K}_{\Omega_1\times\Omega_2}(z_1, z_2, w_1, w_2)= \mathcal{K}_{\Omega_1}(z_1, w_1)\mathcal{K}_{\Omega_2}(z_2, w_2). 
\end{equation*}
\end{lemma}
\begin{proof}
First, note that $(z_1, z_2)\mapsto \mathcal{K}_{\Omega_1}(z_1,w_1)\mathcal{K}_{\Omega_2}(z_2, w_2)$ is holomorphic as the product of holomorphic functions. Moreover, for $w_1\in \Omega_1, w_2\in \Omega_2$, by (\ref{tonelli}),
\begin{equation*}
    \int_{\Omega_1\times \Omega_2}\!\!\!|\mathcal{K}_{\Omega_1}(z_1, w_1) \mathcal{K}_{\Omega_2}(z_2, w_2)|^2\dif\lambda(z_1, z_2)= \int_{\Omega_1}\!\!|\mathcal{K}_{\Omega_1}(z_1, w_1)|^2\dif\lambda(z_1)\int_{\Omega_2}\!\!|\mathcal{K}_{\Omega_2}(z_2, w_2)|^2\dif\lambda(z_2), 
\end{equation*}
i.e., $(z_1, z_2)\mapsto \mathcal{K}_{\Omega\times\Omega_2}(z_1, z_2, w_1, w_2)\in L^2(\Omega_1\times \Omega_2)$. As a result, for $f\in A^2(\Omega_1\times \Omega_2)$, by Cauchy--Schwartz, the pairing $\langle f, \mathcal{K}_{\Omega_1}(\cdot, w_1) \mathcal{K}_{\Omega_2}(\cdot, w_2)\rangle_{L^2(\Omega_1\times \Omega_2)}$ converges in $\C$ and by (\ref{f2}),
\begin{equation*}
    \begin{gathered}
        \int_{\Omega_1\times \Omega_2}f(z_1, z_2)\overline{\mathcal{K}_{\Omega_1}(z_1, w_1)} \overline{\mathcal{K}_{\Omega_2}(z_2, w_2)}\dif\lambda(z_1, z_2)\\ 
        = \int_{\Omega_1}\left( \int_{\Omega_2} f(z_1, z_2)\overline{K_{\Omega_2}(z_2, w_2)}\dif\lambda(z_2)\right)\overline{\mathcal{K}_{\Omega_1}(z_1, w_1)}\dif\lambda(z_1)\\
        = \int_{\Omega_1}f(z_1, w_2)\overline{\mathcal{K}_{\Omega_1}(z_1, w_1)}\dif\lambda(z_1)= f(w_1, w_2). 
    \end{gathered}
\end{equation*}
The same reproducing property holds, by definition, for $\mathcal{K}_{\Omega_1\times \Omega_2}$. By the uniqueness of the Bergman kernel (as in (\ref{naz_eq200})) the claim follows.
\end{proof}

\begin{remark}\label{barycenter_product}
For $K\subset \R^n, L\in \R^m$ convex bodies, let $(x,y)\in \R^{n}\times\R^m$ the coordinates in $\R^{n+m}$. By (\ref{f2}), 
\begin{equation*}
    b(K\times L)=\frac{1}{|K\times L|}\int_{K\times L}(x,y)\dif x\dif y= \frac{1}{|K||L|}\left(|L|\int_K x\dif x, |K| \int_L y\dif y\right), 
\end{equation*}
that is, $b(K\times L)= (b(K), b(L))$. In particular, $b(K\times K)=0$ if and only if $b(K)=0$.
\end{remark}

The next proposition shows how to eliminate subexponential terms in (\ref{nazeq100})
\cite[\S 7]{nazarov}.
\begin{proposition}\label{bk_tensorization}
Let $c>0$, independent of dimension, such that 
$
    \mathcal{B}(K)\geq e^{o(n)} c^n, 
$
for all $n$ and all convex bodies $K\subset \R^n$. Then, 
$
    \mathcal{B}(K)\geq c^n,
$
for all convex bodies.
\end{proposition}
\begin{proof}
Fix $n\in\N$ and a convex body  $K\subset \R^n$.
For each $m\geq 1$, the product $(T_K)^m\defeq \overbrace{T_K\times\ldots\times T_K}^{m\hbox{\sml-times}}= \R^{nm}+ \sqrt{-1}K^m= T_{K^m}$ is the tube domain of $K^m$. By Remark \ref{barycenter_product}, $b(K^m)=b(K)^m$, thus by Lemma \ref{bk_product} and Definition \ref{BKDef},
\begin{equation*}
\begin{aligned}
    \mathcal{B}(K)^m&=\mathcal{K}_{T_K}(\sqrt{-1}b(K),\sqrt{-1}b(K))^m\\
    &= \mathcal{K}_{(T_K)^m}(\sqrt{-1}b(K)^m, \sqrt{-1}b(K)^m)\\&=\mathcal{B}(K^m)= \mathcal{K}_{T_{K^m}}(\sqrt{-1}b(K^m), \sqrt{-1}b(K^m))\geq e^{o(nm)}\frac{c^{nm}}{|K|^{2m}},
\end{aligned}
\end{equation*}
i.e., $\mathcal{B}(K)\geq e^{o(nm)/m} c^n$. Taking $m\to \infty$, by (\ref{little_o}), the claim follows.
\end{proof}

\subsection{Proof of Proposition \ref{nazprop2}}
\label{ProofProp2nd}

Proposition \ref{nazprop2} follows from Lemmas \ref{affine_nazprop2}, \ref{bk_ai} and \ref{bk_tensorization}. 
\begin{proof}[Proof of Proposition \ref{nazprop2}]
For $K\subset \R^n$ a convex body $b(K-b(K))=0$. Therefore, by Lemmas \ref{naz_john_lemma} and \ref{affine_nazprop2}, there exists $A\in GL(n,\R)$ such that $|A(K-b(K))|^2 \mathcal{K}_{T_{A(K-b(K))}}(0,0)\geq e^{o(n)} 4^{-n}$, because $A(K-b(K))\in\mathcal{J}$. 
By Lemma \ref{bk_ai},
\begin{equation}\label{naz_eqd3}
    |K|^2\mathcal{K}_{T_K}(\sqrt{-1}b(K),\sqrt{-1}b(K))= |A(K-b(K))|^2\mathcal{K}_{T_{A(K-b(K))}}(0,0)\geq e^{o(n)}4^{-n}.
\end{equation}
Since (\ref{naz_eqd3}) holds for all convex bodies, by Lemma \ref{bk_tensorization} and Definition \ref{BKDef},
\begin{equation*}
    \mathcal{B}(K)=|K|^2\mathcal{K}_{T_K}(\sqrt{-1}b(K),\sqrt{-1}b(K))\geq 4^{-n},
\end{equation*}
as desired.
\end{proof}

\appendix
\section{Symmetrization}\label{sec_sym}
\label{sym_sec}
In the last line of his paper, Nazarov mentions that while his work should
adapt to non-symmetric bodies \cite[p. 342]{nazarov}, 

\text{ } 

\noindent
\textit{``Unfortunately, this is well-below the bound you can get by the symmetrization trick"}.

\text{ }

\noindent 
This, in conjunction with Remark \ref{NonOptRk},
can be interpreted as follows. Recall the \textit{reflection body} of 
a convex body $K$,
\begin{equation*}
    RK \defeq \mathrm{conv}\{K\cup (-K)\}, 
\end{equation*}
is a symmetric convex body of the same dimension satisfying
\cite[Theorem 3]{rogers-shephard2}:
\begin{theorem}\label{rs}
For $K\subset \R^n$ a convex body satisfying \eqref{intKeq},
$
|RK|\leq 2^n |K|, 
$
with equality if and only if $K$ is a simplex and $0$ is a vertex.
\end{theorem}

\begin{corollary}\label{sym_cor}
Suppose that $\M(K)\geq c^n/n!$ for all symmetric convex bodies $K\subset\R^n$.
Then $\M(K)\geq (c/2)^n/n!$ for all convex bodies $K\subset\R^n$.
\end{corollary}
\begin{proof}
Let $K\subset \R^n$ be a convex body. As in the proof of Theorem \ref{nazarov_lower_bound}, there is no loss in assuming $b(K)=0$, so Theorem \ref{rs} applies.
Because $K\subset RK$,
$
|(RK)^\circ|\leq |K^\circ|,
$
and hence by Theorem \ref{rs}, $\M(K)\ge 2^{-n}\M(RK)$.
Since $RK$ is symmetric then by assumption $\M(RK)\ge c^n/n!$,
so $\M(K)\ge (c/2)^n/n!$, as desired.
\end{proof}

\bigskip
{\sc University of Maryland}

{\tt vmastr@umd.edu, yanir@alum.mit.edu}

\end{document}